\documentclass[12pt,a4paper,onseside]{article}

\usepackage{amsmath,amsthm,amssymb}
\usepackage{fourier}

\newtheorem{theorem}{Theorem}[section]
\newtheorem{corollary}[theorem]{Corollary}
\newtheorem{lemma}[theorem]{Lemma}
\newtheorem{proposition}[theorem]{Proposition}
\theoremstyle{definition}

\numberwithin{equation}{section}

\usepackage{tikz} 
\usepackage{bbm} 
\usepackage{url} 

\allowdisplaybreaks[4] 

\newcommand{\one}{\mathbbm{1}}
\newcommand{\ZZ}{\mathbb{Z}}
\newcommand{\RR}{\mathbb{R}}
\newcommand{\CC}{\mathbb{C}}
\newcommand{\NN}{\mathbb{N}}
\newcommand{\cO}{\mathcal{O}}
\newcommand{\cD}{\mathcal{D}}
\newcommand{\cH}{\mathcal{H}}
\newcommand{\cG}{\mathcal{G}}
\newcommand{\TT}{\mathbb{T}}
\DeclareMathOperator{\supp}{\mathrm{supp}}

\begin{document}

\title{\Large\bf Laplacian eigenvalues with a large negative Robin parameter\\ on a part of the boundary}

\author{Konstantin Pankrashkin\\[\medskipamount]
	\small Carl von Ossietzky Universit\"at Oldenburg, Fakult\"at V, Institut f\"ur Mathematik\\
	\small 26111 Oldenburg, Germany\\[\smallskipamount]
	\small E-Mail: \url{konstantin.pankrashkin@uol.de}, Webpage: \url{http://uol.de/pankrashkin}
}

\date{}

\maketitle

\begin{abstract}
	\noindent We study the Laplacian eigenvalues for smooth planar domains with strongly attractive Robin conditions imposed on a part of the boundary and Neumann condition on the remaining boundary.
	The asymptotic behavior of individual eigenvalues is described in terms of an effective Schr\"odinger-type operator on an interval with boundary conditions at the endpoints. For some typical geometries a more precise 	asymptotics in terms of the boundary curvature is derived.
\end{abstract}


\maketitle

\section{Introduction}

Let $\Omega\subset\RR^d$ be an open set with a reasonably regular boundary $\partial\Omega$ (for example, bounded and Lipschitz) and an outer unit normal $\nu$. By a Robin Laplacian with negative boundary parameter in $\Omega$
one usually means the operator $Q_{p,\alpha}$ in $L^2(\Omega)$ acting as $u\mapsto-\Delta u$ on the functions $u$ satisfying the Robin boundary condition $\partial_\nu u=\alpha p u$, where $p\ge 0$ is a given function (for example, bounded) and $\alpha>0$ is a parameter. The name ``negative'' is justified by the fact that for all $u$ in the operator domain one has
\[
\langle u,Q_{p,\alpha} u\rangle_{L^2(\Omega)}=\int_\Omega|\nabla u|^2\, dx-\alpha\int_{\partial \Omega} p|u|^2\,dS,
\]
where $dS$ means the integration with respect to the hypersurface measure, so that the boundary term
makes a negative contribution for $p\not\equiv 0$  (often this is also termed as an attractive boundary condition). The structure of the above expression suggests that in the limit $\alpha\to+\infty$ the boundary might play a central role in the asymptotic behavior of the eigenvalues $E_n(Q_{p,\alpha})$. The intuitive
expectation was made rigorous by Lacey, Ockedon, Sabina \cite{los} and Levitin, Parnovski \cite{LP}, who showed
that the main term in the asymptotic expansion of the first eigenvalue of $Q_{p,\alpha}$ is determined
by the singularities of the boundary. For smooth domains and $p\equiv 1$ the first eigenvalue behaves always
as $E_1(Q_{1,\alpha})\sim -\alpha^2$, so it is reasonable to look at the higher eigenvalues and
at the next terms in the asymptotics. Pankrashkin \cite{nano13} and Exner, Minakov, Parnovski \cite{emp}
showed that in the two-dimensional case ($d=2$) the next term for individual eigenvalues is determined by the maximum curvature of the boundary. It was asked if it is possible to analyze the gaps between individual eigenvalues as well.  A major step in this direction
was made by Helffer, Kachmar \cite{HK}. Roughly speaking, they observed that for any fixed number $n\in\NN$
one has $E_n(Q_{1,\alpha})=-\alpha^2+E_n(\Lambda_\alpha) +\text{small relative error}$ with an effective operator $\Lambda_\alpha$ on the boundary given by $\Lambda_\alpha:=-\partial_s^2-\alpha k$, where $\partial_s$ is the differentiation with respect to the arclength and $k$ is the curvature. Formally,
they considered the case of a curvature having a single non-degenerate maximum, which lead to the eigenvalue spacing of order $\sqrt{\alpha}$. The idea of an effective operator and a semiclassical reduction were then extended to a variety of situations including the multi-dimensional  case \cite{PP}, Weyl asymptotics \cite{kkr},
tunneling problems \cite{HKR,hp}, domains with boundary singularities \cite{kobp, PV} and cusps \cite{kop2,vogel}. The study of Robin Laplacians also lead to important advances
in isoperimetric spectral problems \cite{DFNT,fk,kl} and added new ingredients to the analysis of magnetic operators \cite{fahs,HKR2}. We refer to the reviews in \cite{DFK,kobp} for a summary of most available results.

While a significant amount of work was done on the analysis of the case $p\equiv 1$, no detailed asymptotic results seem known  for non-constant $p$. In the present work we are making a first step in this direction
by considering the case when $p$ is the indicator function of a subset $\Gamma\subset\partial\Omega$. This corresponds to the situation when the parameter-dependent Robin condition is imposed on $\Gamma$ only, while the rest of the boundary is endowed with Neumann condition. From now on let $d=2$ and $\Omega\subset \RR^2$ be a simply connected domain with a $C^4$-smooth boundary $\partial \Omega$ of length $L$.
Let $\Gamma\subset\partial\Omega$ be a open connected set of length $\ell\in (0,L)$.
For $\alpha\in\RR$, denote by $Q_\alpha$ the self-adjoint operator in $L^2(\Omega)$
acting as $Q^\Omega_\alpha u =-\Delta u$ on the functions $u$ satisfying the boundary condition
\[
\partial_\nu u =\alpha u \text{ on } \Gamma,
\quad
\partial_\nu u =0 \text{ on } \partial\Omega\setminus{\overline \Gamma},
\]
where $\nu$ is the outer unit normal at $\partial\Omega$. More precisely, $Q^\Omega_\alpha$ is
the self-adjoint operator in $L^2(\Omega)$ associated with
the closed hermitian sesquilinear form $q_\alpha$ defined on
$\cD(q_\alpha)=H^1(\Omega)$ by
\[
q_\alpha(u,u)=\int_{\Omega} |\nabla u|^2\, dx -\alpha \int_\Gamma |u|^2\,dS,
\]
where $d S$ stands for the integration with respect to the arclength.
The operator $Q_\alpha$ has compact resolvent, and for each fixed $n\in\NN$ we are interested in the asymptotic behavior
of its $n$-th eigenvalue $E_n(Q_\alpha)$ for $\alpha\to+\infty$.

\begin{figure}
	
	\tikzset{every picture/.style={line width=0.75pt}} 
	
	\centering
	\scalebox{0.7}{\begin{tikzpicture}[x=0.75pt,y=0.75pt,yscale=-1,xscale=1]
			
			\draw  [color={rgb, 255:red, 155; green, 155; blue, 155 }  ,draw opacity=1 ][fill={rgb, 255:red, 155; green, 155; blue, 155 }  ,fill opacity=0.34 ][line width=2.25]  (121.5,86) .. controls (141.5,76) and (87.5,-16) .. (222.5,35) .. controls (357.5,86) and (299.5,113) .. (248.5,132) .. controls (197.5,151) and (57.5,139) .. (37.5,109) .. controls (17.5,79) and (101.5,96) .. (121.5,86) -- cycle ;
			\draw [line width=3]    (137.5,24) .. controls (176.5,2) and (314.5,70) .. (304.5,91) ;
			\draw  [fill={rgb, 255:red, 0; green, 0; blue, 0 }  ,fill opacity=1 ] (298,91) .. controls (298,87.41) and (300.91,84.5) .. (304.5,84.5) .. controls (308.09,84.5) and (311,87.41) .. (311,91) .. controls (311,94.59) and (308.09,97.5) .. (304.5,97.5) .. controls (300.91,97.5) and (298,94.59) .. (298,91) -- cycle ;
			\draw  [fill={rgb, 255:red, 0; green, 0; blue, 0 }  ,fill opacity=1 ] (131,24) .. controls (131,20.41) and (133.91,17.5) .. (137.5,17.5) .. controls (141.09,17.5) and (144,20.41) .. (144,24) .. controls (144,27.59) and (141.09,30.5) .. (137.5,30.5) .. controls (133.91,30.5) and (131,27.59) .. (131,24) -- cycle ;
			\draw [line width=1.5]    (46.75,91) -- (40.24,51.46) ;
			\draw [shift={(39.75,48.5)}, rotate = 80.65] [color={rgb, 255:red, 0; green, 0; blue, 0 }  ][line width=1.5]    (14.21,-4.28) .. controls (9.04,-1.82) and (4.3,-0.39) .. (0,0) .. controls (4.3,0.39) and (9.04,1.82) .. (14.21,4.28)   ;
			\draw  [fill={rgb, 255:red, 0; green, 0; blue, 0 }  ,fill opacity=1 ] (40.25,91) .. controls (40.25,87.41) and (43.16,84.5) .. (46.75,84.5) .. controls (50.34,84.5) and (53.25,87.41) .. (53.25,91) .. controls (53.25,94.59) and (50.34,97.5) .. (46.75,97.5) .. controls (43.16,97.5) and (40.25,94.59) .. (40.25,91) -- cycle ;
			\draw [color={rgb, 255:red, 155; green, 155; blue, 155 }  ,draw opacity=1 ]   (105.5,88.5) -- (82.25,90) ;
			\draw [shift={(82.25,90)}, rotate = 356.31] [color={rgb, 255:red, 155; green, 155; blue, 155 }  ,draw opacity=1 ][line width=0.75]    (17.64,-3.29) .. controls (13.66,-1.4) and (10.02,-0.3) .. (6.71,0) .. controls (10.02,0.3) and (13.66,1.4) .. (17.64,3.29)(10.93,-3.29) .. controls (6.95,-1.4) and (3.31,-0.3) .. (0,0) .. controls (3.31,0.3) and (6.95,1.4) .. (10.93,3.29)   ;
			\draw [color={rgb, 255:red, 155; green, 155; blue, 155 }  ,draw opacity=1 ]   (90.25,133.5) -- (111.75,137) ;
			\draw [shift={(111.75,137)}, rotate = 189.25] [color={rgb, 255:red, 155; green, 155; blue, 155 }  ,draw opacity=1 ][line width=0.75]    (17.64,-4.9) .. controls (13.66,-2.3) and (10.02,-0.67) .. (6.71,0) .. controls (10.02,0.67) and (13.66,2.3) .. (17.64,4.9)(10.93,-4.9) .. controls (6.95,-2.3) and (3.31,-0.67) .. (0,0) .. controls (3.31,0.67) and (6.95,2.3) .. (10.93,4.9)   ;
			\draw [color={rgb, 255:red, 155; green, 155; blue, 155 }  ,draw opacity=1 ]   (243.25,133.5) -- (265.75,125.5) ;
			\draw [shift={(265.75,125.5)}, rotate = 160.43] [color={rgb, 255:red, 155; green, 155; blue, 155 }  ,draw opacity=1 ][line width=0.75]    (17.64,-4.9) .. controls (13.66,-2.3) and (10.02,-0.67) .. (6.71,0) .. controls (10.02,0.67) and (13.66,2.3) .. (17.64,4.9)(10.93,-4.9) .. controls (6.95,-2.3) and (3.31,-0.67) .. (0,0) .. controls (3.31,0.67) and (6.95,2.3) .. (10.93,4.9)   ;
			
			\draw (318,82.4) node [anchor=north west][inner sep=0.75pt]    {$\gamma ( 0)$};
			\draw (94.5,14.4) node [anchor=north west][inner sep=0.75pt]    {$\gamma ( \ell )$};
			\draw (256,20.4) node [anchor=north west][inner sep=0.75pt]  [font=\large]  {$\Gamma $};
			\draw (48.75,94.4) node [anchor=north west][inner sep=0.75pt]    {$\gamma ( s)$};
			\draw (46.25,34.4) node [anchor=north west][inner sep=0.75pt]    {$\nu ( s)$};
			\draw (174.5,76.4) node [anchor=north west][inner sep=0.75pt]  [font=\Large]  {$\Omega $};

		\end{tikzpicture}
}
	
	\caption{The domain $\Omega$ and the set $\Gamma\subset\partial\Omega$.}\label{fig1}
	
\end{figure}
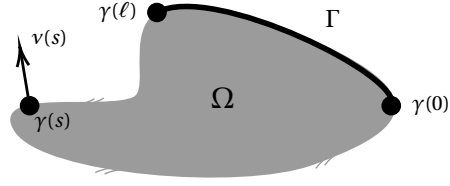

Let $\gamma:[0,L]\to\RR^2$ be an arclength parametrization of $\partial\Omega$, then $\gamma$ extends to an $L$-periodic $C^4$-smooth function on $\RR$. The outer unit normal to $\partial\Omega$ at the point $\gamma(s)$ will be denoted by $\nu (s)$. By a suitable choice of the starting point $\gamma(0)$ and the orientation we may assume that  $\Gamma=\gamma\big((0,\ell)\big)$ and  $\det \big(\nu(\cdot),\gamma'(\cdot)\big)=1$, see Figure~\ref{fig1}. Then the curvature $k(s)$ of $\partial\Omega$ at the point $\gamma(s)$ is defined by $\gamma''(s)=-k(s)\nu(s)$, which is equivalent to $\nu'(s)=k(s)\gamma'(s)$ due to Frenet formulas.
Note that $k$ is $C^m$ if the boundary is $C^{m+2}$-smooth (so $k$ is at least $C^2$ in our case).

The main results of the present paper can be very roughly summarized as follows: For any fixed $n\in\NN$ and $\alpha\to+\infty$ one has
$E_n(Q_\alpha)=-\alpha^2+E_n(L_\alpha)+\text{small relative error}$,
where $L_\alpha$ is the ``effective operator'' in $L^2(0,\ell)$ given by $f\mapsto-f''-\alpha k$
\emph{with Dirichlet boundary conditions} at the endpoints (for a rigorous formulation with precise remainder estimates we refer to Theorem \ref{thmeff}). It should be noted that if $k$ attains its maximum in the interior
of $\Gamma$, the localization argument of \cite{HK} can be directly transferred, and we are mainly interested
in the case when the maximum curvature is attained at an endpoint of $\Gamma$. The main difficulty to overcome comes from the contradictory expectations that (i) the eigenfunctions should be localized near the endpoint and (ii) at the same time the value of the eigenfunction is very small at the endpoint itself. This is an essentially new feature when compared with the case $\Gamma=\partial\Omega$ which does not require any boundary conditions as $\partial\Omega$ is a closed curve. Using standard approaches of the semiclassical analysis to analyze the effective operator we then give more detailed results for several specific situations. For the formulations it will be convenient to denote
\[
k_*:=\sup_{s\in(0,\ell)} k(s), \text{ i.e. the maximum curvature on $\Gamma$,}
\]
then, in particular, the following holds true:
\begin{itemize}
	\item if $k$ is constant on $(0,\ell)$, then 
	$E_n(Q_\alpha)=-\alpha^2-k_*\alpha-\frac{k_*^2}{2}+\frac{\pi^2 n^2}{\ell^2}+o(1)$,
	see Theorem \ref{kconst},
	\item if $0$ is a strict maximum of $k$ on $[0,\ell]$ with  $k'(0)<0$, then
	\[
	E_n(Q_\alpha)=-\alpha^2-k_*\alpha +a_n\big(-k'(0)\big)^\frac{2}{3} \alpha^{\frac{2}{3}}+o(\alpha^{\frac{2}{3}}),
	\]
	where $(-a_n)$ is the $n$-th zero of the Airy function $\mathrm{Ai}$, see Theorem \ref{kmax}.
\end{itemize}
In fact, the main results also cover degenerate maxima attained at endpoints and maxima attained in interior points
and contain more precise remainder estimates: we refer to Theorems \ref{kmax0} and \ref{kmax} for detailed formulations. Nevertheless we do not expect that our remainder estimates are optimal.

Let us now describe the structure of the paper. Our first main goal is to describe the asymptotic behavior of $E_n(Q_\alpha)$ for $\alpha\to+\infty$ in terms of an effective operator $\Lambda'_\alpha$ on $(0,\ell)$ given by \eqref{ll1}.
In Section~\ref{sec2} we introduce the $r$-neighborhood $\Omega_r$ of $\partial\Omega$
and consider a ``truncated version'' $Q_{\alpha,r}$ of $Q_\alpha$ acting in $L^2(\Omega_r)$. This new operator
is then rewritten in tubular coordinates near the boundary, which gives rise to a unitarily equivalent
operator $P_{\alpha,r}$ on $\TT\times(0,r)$, with $\TT:=\RR/(L\ZZ)$. By adjusting the coefficients in $P_{\alpha,r}$ we construct two operators $P^\pm_{\alpha,r}$ in $\TT\times(0,r)$ whose eigenvalues
provide lower/upper bounds for those of of $P_{\alpha,r}$. In Proposition \ref{propup}
we introduce special test functions for $P^+_{\alpha,r}$, which gives an upper bound for its eigenvalues
and, in turn, an upper bound for $E_n(Q_\alpha)$ in terms of $\Lambda'_\alpha$. 
In Section \ref{sec3} we are moving towards the lower bound. First, the upper bound from Section \ref{sec2}
is used to show that the individual eigenfunctions of $Q_\alpha$ are localized near the boundary,
and this shows that the eigenvalues of $Q_\alpha$ are very close to those of $Q_{\alpha,\alpha^{-\sigma}}$ with $\sigma\in(0,1)$, see Corollary \ref{cor4}. We then obtain a lower bound for the eigenvalues of $Q_{\alpha,\alpha^{-\sigma}}$ in terms of an intermediate operator $\Lambda_{\alpha,\rho}$ on $\TT$ defined in \eqref{ll2}: the main part is Lemma \ref{propdown} collecting various estimates for the associated sesquilinear forms and based on the Born-Oppenheimer-type asymptotic separation of variables, and the final result on the eigenvalue comparison is given in Proposition \ref{prop36}. All preceding constructions are then summarized in Corollary \ref{cor}: at this moment we have an upper bound in terms of $\Lambda'_\alpha$ and a lower bound in terms of $\Lambda_{\alpha,r}$. In Section \ref{sec4} we show that
actually the eigenvalues of these two one-dimensional operators are close to each other. This part the argument is based
on an identification technique we learned from \cite{Ex03} (see Proposition \ref{Comp}), and the final result is given in Lemma \ref{lem11}. Similar constructions were used in \cite{kobp} for Robin laplacians in polygons.
We remark that the worst term in the final remainder arises in this step. As a summary of all preceding computations we obtain Theorem \ref{thmeff} describing the asymptotics of $E_n(Q_\alpha)$ solely in terms of $\Lambda'_\alpha$. The operator $\Lambda'_\alpha$ is a semiclassical 
Schr\"odinger operator, and by applying several standard approaches in Section \ref{sec5} we obtain 
more precise results on its eigenvalues in terms of the curvature, which then translates into Theorems \ref{kconst}, \ref{kmax0} and  \ref{kmax} describing the asymptotic behavior of the eigenvalues of $Q_\alpha$
for several typical geometric situations.

\section{Tubular coordinates and the upper bound}\label{sec2}

Denote $\TT:=\RR/(L\ZZ)$ and $\Pi_r:=\TT\times(0,r)$ for $r>0$.
The tubular neighborhood theorem from the differential geometry states that there exists $R>0$ with $\|k\|_\infty R<1$ such that the map $\Phi: \Pi_R\ni (s,t)\mapsto \gamma(s)-t\nu(s)\in\RR^2$
is injective. Moreover, for each $r\in(0,R)$ the map $\Phi$ defines a diffeomorphism between $\Pi_r$ and the domain
\[
\Omega_r:=\big\{
x\in\Omega:\ d_{\partial\Omega}(x)<r
\big\}, \quad d_{\partial\Omega}(x):=\min_{y\in \partial\Omega} |x-y|.
\]
In addition,  for any $(s,t)\in \Pi_R$ one has $d_{\partial\Omega}\big(\Phi(s,t)\big)=t$, and the set
$\partial\Omega_r\setminus\partial\Omega\equiv \big\{x\in\Omega:\, d_{\partial\Omega}(x)=r\big\}$
is a $C^2$-smooth closed curve.

For $r\in(0,R)$ we define a closed hermitian sesquilinear form $q_{\alpha,r}$ in $L^2(\Omega)$,
\begin{align*}
q_{\alpha,r}(u,u)&:=\int_{\Omega_r} |\nabla u|^2\, dx -\alpha \int_\Gamma |u|^2\,dS,\\
\cD(q_{\alpha,r})&:=\big\{ u \in H^1(\Omega_r):\ u=0 \text{ on } \partial\Omega_r\setminus\partial\Omega\big\}=:\Tilde H^1_0(\Omega_r)
\end{align*}
and let $Q_{\alpha,r}$ be the associated self-adjoint operator in $L^2(\Omega_r)$.
For each $u\in \cD(q_{\alpha,r})$ its extension $\Tilde u$ by zero on $\Omega$ belongs to $\cD(q_\alpha)$ and satisfies
$\|\Tilde u\|_{L^2(\Omega)}=\|u\|_{L^2(\Omega_r)}$ and $q_{\alpha}(\Tilde u,\Tilde u)=q_{\alpha,r}(u,u)$,
so due to the min-max principle there holds
\begin{equation}
	\label{qqar1}
	E_n(Q_{\alpha})\le E_n(Q_{\alpha,r}) \text{ for all $n\in\NN$, $\alpha>0$, $r\in(0,R)$}.
\end{equation}

For any $r\in(0,R)$ consider the unitary operator
\[
\Theta_r: L^2(\Omega)\to L^2(\Pi_r),
\quad
(\Theta_r u)(s,t):=\sqrt{1-tk(s)}\,u\big(\Phi(s,t)\big),
\]
and the closed hermitian sesquilinear form $p_{\alpha,r}$ in $L^2(\Pi_r)$ given by
\begin{align*}
	p_{\alpha,r}(v,v)&=\int_{\Pi_r} \Bigg(\dfrac{1}{\big(1-t k(s)\big)^2} \big|\partial_s v(s,t)\big|^2 +\big|\partial_t v(s,t)\big|^2 - V(s,t) \big|v(s,t)\big|^2\Bigg)\, ds\, dt\\
	&\qquad
	-\int_0^\ell \Big(\alpha+\frac{k(s)}{2}\Big)\big|v(s,0)\big|^2\,ds,\\
	V(s,t)&:=\dfrac{t k''(s)}{2\big(1-t k(s)\big)^3}+\dfrac{5t^2 k'(s)^2}{4\big(1-t k(s)\big)^4}+\dfrac{\kappa(s)^2}{4\big(1-t k(s)\big)^2},\\
	\cD(p_{\alpha,r})&=\big\{ v\in H^1(\Pi_r):\ v(\cdot,r)=0\big\}=:\Tilde H^1_0(\Pi_r).
\end{align*}

A standard computation shows that $\cD(p_{\alpha,r})=\Theta_r \cD(q_{\alpha,r})$ and that for any $u\in \cD(q_{\alpha,r})$ there holds
$q_{\alpha,r}(u,u)=p_{\alpha,r}(\Theta_r u,\Theta_r u)$; see e.g. \cite[Section~2]{EY}.
It follows that the self-adjoint operator $P_{\alpha,r}$ in $L^2(\Pi_r)$ generated by the form $p_{\alpha,r}$ is unitarily equivalent to $Q_{\alpha,r}$, in particular,
\begin{equation}
	\label{qqar2}
	E_n(Q_{\alpha,r})= E_n(P_{\alpha,r}) \text{ for all $n\in\NN$, $\alpha>0$, $r\in(0,R)$}.
\end{equation}

We note that for a suitably chosen $A>0$ one has the estimates
\begin{equation}
	\label{vst}
	\begin{aligned}	
		\big|V(s,t)-&\dfrac{k(s)^2}{4}\big|\le A t,\qquad \Big|\dfrac{1}{\big(1-t k(s)\big)^2}-1\Big|\le At \quad	\text{ for all $(s,t)\in\Pi_r$ and $r\in(0,R)$.}
	\end{aligned}\quad 
\end{equation}
In particular, $p^-_{\alpha,r}(v,v)\le p_{\alpha,r}(v,v)\le p^+_{\alpha,r}(v,v)$ for all $v\in \Tilde H^1_0(\Pi_r)$,
where $p^\pm_{\alpha,r}$ are the closed hermitian sesquilinear forms in $L^2(\Pi_r)$ defined on the domain $\Tilde H^1_0(\Pi_r)$ by
\begin{align*}
	p^\pm_{\alpha,r}&(v,v)=\int_{\Pi_r} \Bigg((1\pm A r) \big|\partial_s v(s,t)\big|^2 +\big|\partial_t v(s,t)\big|^2\\
	&\qquad +\Big(\pm Ar-\dfrac{\kappa(s)^2}{4}\Big) \big|v(s,t)\big|^2\Bigg)\, ds\, dt-\int_0^\ell \Big(\alpha+\frac{k(s)}{2}\Big)\big|v(s,0)\big|^2\,ds,
\end{align*}
If $P^\pm_{\alpha,r}$ are the self-adjoint operators in $L^2(\Pi_r)$ generated by the forms $p^\pm_{\alpha,r}$, then the min-max principle implies
\begin{equation}
	\label{qqar3}
	E_n(P^-_{\alpha,r})\le E_n(P_{\alpha,r})\le E_n(P^+_{\alpha,r})  \text{ for all $n\in\NN$, $\alpha>0$, $r\in(0,R)$}.
\end{equation}	
By combining \eqref{qqar3} with \eqref{qqar1} and \eqref{qqar2} we conclude, in particular, that
\begin{equation}
	\label{qqar4}
	E_n(Q_\alpha)\le E_n(P^+_{\alpha,r})  \text{ for all $n\in\NN$, $\alpha>0$, $r\in(0,R)$}.
\end{equation}	

To analyze the eigenvalues of $P^\pm_{\alpha,r}$ we employ the following result from \cite[Lemma~2.1]{PP}:
\begin{lemma}\label{lem1d}
	For $\alpha>0$ and $r>0$, denote by $T_{\alpha, r}$ the operator $f\mapsto -f''$
	acting in $L^2(0,r)$ on the domain
	\[
	\cD(T_{\alpha,r})=\big\{
	f\in H^2(0,r):\  f'(0)=-\alpha f(0),\  f(r)=0
	\},
	\]
	which is generated by the closed hermitian sesquilinear form
	\begin{align}
	t_{\alpha,r}(f,f)&=\int_0^r\big|f'(t)\big|^2\,dt -\alpha \big|f(0)\big|^2, \quad
	\cD(t_{\alpha,r})=\big\{f\in H^1(0,r):\, f(r)=0\big\}=:\Tilde H^1_0(0,r).\label{h10}
	\end{align}
	
	If $r\alpha$ tends to $+\infty$, then $T_{\alpha,r}$
	has a unique negative eigenvalue, and
	$	E_1(T_{\alpha,r})=-\alpha^2 + \cO(\alpha^2e^{-r\alpha})$  for $r\alpha\to+\infty$.
	Furthermore, if $\psi$ is an associated $L^2$-normalized eigenfunction, then $\big|\psi(0)\big|^2=2\alpha+ \cO(\alpha e^{-r\alpha})$ for $r\alpha\to+\infty$.
\end{lemma}

We will denote by $\Lambda'_\alpha$ the self-adjoint Schr\"odinger operator
\begin{equation}
	\label{ll1}
\Lambda'_\alpha:\ f\mapsto -f''+\Big(\alpha (k_*-k)+\frac{k_*^2-2kk_*-k^2}{4}\Big)f
\end{equation}
with Dirichlet boundary conditions in $L^2(0,\ell)$.

\begin{proposition}\label{propup}
Let $\sigma\in(0,1)$, then
	there are $A,B,\alpha_0>0$ such
		that for any $n\in \NN$ and any $\alpha>\alpha_0$ there holds
		\begin{equation}
			\label{enq}
		E_n(Q_\alpha)\le -\alpha^2 -\alpha k_* +(1+A\alpha^{-\sigma})E_n(\Lambda'_\alpha) + B\alpha^{-\sigma}.
		\end{equation}
\end{proposition}

\begin{proof}
Let $\psi$ be an $L^2$-normalized eigenfunction for the first eigenvalue of $T_{\alpha+\frac{k_*}{2},r}$ (Lemma \ref{lem1d}).
For $g\in H^1_0(0,\ell)\subset H^1(\TT)$ consider the function
$v:=g\otimes \psi:\ \Pi_r\ni (s,t)\mapsto g(s)\psi(t)$,
then $v\in \Tilde H^1_0(\Pi_r)$ with $\|v\|_{L^2(\Pi_r)}=\|g\|_{L^2(0,\ell)}$ and
\begin{align*}
	p^+_{\alpha,r}(v,v)&=(1+Ar)\int_0^\ell \big|g'(s)\big|^2\,ds+\int_0^\ell \big|g(s)\big|^2\,ds\int_0^r \big|\psi'(t)\big|^2\,dt\\
	&+\int_0^\ell \Big(Ar-\dfrac{k(s)^2}{4}\Big)\big|g(s)\big|^2
	-\big|\psi(0)\big|^2 \int_0^\ell \Big(\alpha+\frac{k(s)}{2}\Big)\big|g(s)\big|^2\,ds.
\end{align*}
Due to the choice of $\psi$ there holds
\[
\int_0^r \big|\psi'(t)\big|^2\,dt-\Big(\alpha+\frac{k_*}{2}\Big) \big|\psi(0)\big|^2=E_1(T_{\alpha+\frac{k_*}{2},r}),
\]
therefore,
\begin{align*}
p^+_{\alpha,r}(v,v)&=\int_0^\ell \Bigg[(1+Ar)\big|g'(s)\big|^2+ \Big(E_1(T_{\alpha+\frac{k_*}{2},r})+Ar-\dfrac{k(s)^2}{4}+\big|\psi(0)\big|^2\dfrac{k_*-k(s)}{2}\Big)\big|g(s)\big|^2\Bigg]\,ds.
\end{align*}
Now set $r:=\alpha^{-\sigma}$ with $\sigma\in(0,1)$. Due to Lemma \ref{lem1d} there exist $\alpha_0>0$ and $b>0$ such that for all $\alpha>\alpha_0$ we have
\begin{align*}
	E_1(T_{\alpha+\frac{k_*}{2},\alpha^{-\sigma}})&\le-\Big(\alpha+\frac{k_*}{2}\Big)^2+b\alpha^{-\sigma},\quad
	\big|\psi(0)\big|^2\dfrac{k_*-k(s)}{2}\le \Big(\alpha+\frac{k_*}{2}\Big) \big(k_*-k(s)\big)-b\alpha^{-\sigma} \text{ for all }s\in(0,\ell),
\end{align*}
which leads to
\begin{align*}
	p^+_{\alpha,\alpha^{-\sigma}}(v,v)&\le (-\alpha-\alpha k_*)\|g\|^2_{L^2(0,\ell)}\\
	&+ (1+A\alpha^{-\sigma})\int_0^\ell \bigg[ \big|g'(s)\big|^2+\Big(\alpha \big(k_*-k(s)\big) +\frac{k_*^2-2kk_*-k^2}{4}\Big)\big|g(s)\big|^2\bigg]\,ds
	+B\alpha^{-\sigma}\|g\|^2_{L^2(0,\ell)}\\
	&\text{for all $g\in H^1_0(0,\ell)$ and $\alpha>\alpha_0$, with a suitable $B>0$.}
\end{align*}
The integral of the right-hand side is exactly the sesquilinear form for the operator $\Lambda'_\alpha$
computed on $(g,g)$. Therefore, the substition of the above functions $v$ into the min-max principle shows that 
for all $n\in\NN$ and $\alpha>\alpha_0$ one has
$E_n(P^+_{\alpha,\alpha^{-\sigma}})\le -\alpha^2-\alpha k_* +(1+A\alpha^{-\sigma})E_n(\Lambda'_\alpha)
+ B\alpha^{-\sigma}$.
Due to \eqref{qqar3} for all sufficiently large $\alpha$ we have
$E_n(Q_\alpha)\le E_n(P^+_{\alpha,\alpha^{-\sigma}})$ for all $n\in\NN$,
which gives the sought estimate.
\end{proof}

\begin{corollary}\label{coro}
	For any fixed $n\in\NN$ one has $E_n(Q_\alpha)=\cO(\alpha^2)$ for $\alpha\to+\infty$.
\end{corollary}

\begin{proof}
The upper bound is proved in Proposition \ref{propup}.
By \cite[Theorem 1.5.1.10]{Gri} one can find some $c>0$ such that for all $\varepsilon\in(0,1)$ and $u\in H^1(\Omega)$ there holds
\[
\int_\Gamma |u|^2\,ds\le \int_{\partial\Omega} |u|^2\,ds\le \varepsilon\int_\Omega |\nabla u|^2\, dx+\frac{c}{\varepsilon} \int_\Omega |u|^2\, dx.
\]
For $\varepsilon:=\frac{1}{\alpha}$ one arrives at $q_\alpha(u,u)\ge -c\alpha^2\|u\|^2_{L^2(\Omega)}$, which gives the lower bound $E_1(Q_\alpha)\ge-c\alpha^2$.
\end{proof}

\section{Lower bound: Effective operator on $\TT$}\label{sec3}

For each fixed $n$, the equation \eqref{enq} gives $E_n(Q_\alpha)\le -\alpha^2+\cO(\alpha)$ for $\alpha\to+\infty$. A minor adaptation of \cite[Theorem~5.1]{HK} (which formally considered the case $\Gamma=\partial\Omega$) gives the following Agmon-type estimate:
\begin{lemma}[Boundary localization]\label{prop-agmon}
Let $n\in\NN$ be fixed and $u_\alpha$ be an $L^2$-normalized eigenfunction of $Q_\alpha$ for the eigenvalue $E_n(Q_\alpha)$. Then for any $\theta\in(0,1)$ there are $C>0$ and $\alpha_0>0$ such that
\begin{equation}
	\label{agmon1}
\int_\Omega \Big(\alpha^{-2}\big|\nabla u_\alpha (x)\big|^2+\big|u_\alpha(x)\big|^2\big) e^{2\theta\alpha d_{\partial\Omega}(x)}\,dx\le C \text{ for all $\alpha>\alpha_0$}.
\end{equation}
\end{lemma}

\begin{corollary}\label{cor4}
Let $\sigma\in(0,1)$, then for any fixed $n\in \NN$ and $M>0$ there holds
$
E_n(Q_\alpha)=E_n(Q_{\alpha,\alpha^{-\sigma}})+\cO(\alpha^{-M})$  for $\alpha\to+\infty$.
\end{corollary}

\begin{proof}
The proof is very standard if one takes into account Lemma \ref{agmon1}, but we include it for the sake of completeness.
In view of the inequality \eqref{qqar1} it is sufficient to find a suitable upper bound for $E_n(Q_{\alpha,\alpha^{-\sigma}})$
in terms of $E_n(Q_\alpha)$.

Let $u_{1,\alpha},\dots,u_{n,\alpha}$ be eigenfunctions of $Q_\alpha$ for $E_1(Q_\alpha),\dots, E_n(Q_\alpha)$
forming an orthonormal system in $L^2(\Omega)$. Denote
$U_\alpha:=\mathop{\mathrm{span}}(u_{1,\alpha},\dots,u_{n,\alpha})$.
As Lemma \ref{prop-agmon} holds for each $u_{j,\alpha}$, we conclude that for some $\theta\in(0,1)$ and $C>0$ there holds
\begin{equation}
	\label{agmon2}
	\int_\Omega \Big(\alpha^{-2}\big|\nabla u (x)\big|^2+\big|u(x)\big|^2\big) e^{2\theta\alpha d_{\partial\Omega}(x)}\,dx\le C\|u\|^2_{L^2(\Omega)} 
\end{equation}
for all $\alpha>\alpha_0$ and all $u\in U_\alpha$.

Let $\chi:\RR\to\RR$ be a $C^\infty$ function with $0\le \chi\le 1$ and such that $\chi(t)=1$ for $|t|\le\frac{1}{2}$ and $\chi(t)=0$ for $|t|\ge 1$. Consider the linear map
\[
J:\, U_\alpha\to \cD(q_{\alpha,\alpha^{-\sigma}}),
\quad
(J u)(x):=\chi\big(\alpha^\sigma d_{\partial \Omega}(x)\big)u(x).
\]
Due to the choice of $\chi$ we have
$J u=u$  in $\Omega_{\frac{\alpha^{-\sigma}}{2}}$
for all $u\in U_\alpha$. Furthermore, for any $u\in U_\alpha$ and $v:=Ju$ and all sufficiently large $\alpha$ there holds
\begin{align*}
	\|u\|^2_{L^2(\Omega)}&-\|v\|^2_{L^2(\Omega_{\alpha^{-\sigma}})}=\int_{\Omega\setminus \Omega_{\frac{\alpha^{-\sigma}}{2}}} \Big( 1-\chi\big(\alpha^\sigma d_{\partial \Omega}(x)\big)\Big)^2 \big|u(x)\big|^2\,dx\\
	&\le \int_{\Omega\setminus \Omega_{\frac{\alpha^{-\sigma}}{2}}}  \big|u(x)\big|^2\,dx
	\le \int_{\Omega\setminus \Omega_{\frac{\alpha^{-\sigma}}{2}}}  e^{2\theta \alpha \big(d_{\partial\Omega}(x)-\frac{\alpha^{-\sigma}}{2}\big)}\big|u(x)\big|^2\,dx\le e^{-\theta \alpha^{1-\sigma}}\int_{\Omega}  e^{2\theta \alpha d_{\partial\Omega}(x)}\big|u(x)\big|^2\,dx\\
&\stackrel{\eqref{agmon2}}{\le}Ce^{-\theta \alpha^{1-\sigma}}\,\|u\|^2_{L^2(\Omega)}\le \alpha^{-N}\|u\|^2_{L^2(\Omega)}
\end{align*}
with an arbitrarily chosen $N>0$. Therefore,
\begin{equation}
	\label{avn}
\big(1-\alpha^{-N}\big)\|u\|^2\le \|v\|^2\le \big(1+\alpha^{-N}\big)\|u\|^2,
\end{equation}
so $J$ is injective, $\dim J(U_\alpha)=n$ for all sufficiently large $\alpha$.
Due to $u=v$ on $\Omega_{\frac{\alpha^{-\sigma}}{2}}$ and on $\Gamma$
 we have
\[
q_{\alpha,\alpha^{-\sigma}}(v,v)-q_\alpha(u,u)=\int_{\Omega_{\alpha^{-\sigma}}\setminus \Omega_{\frac{\alpha^{-\sigma}}{2}}} |\nabla v|^2\, dx-\int_{\Omega\setminus \Omega_{\frac{\alpha^{-\sigma}}{2}}} |\nabla u|^2\,dx.
\]
We estimate as before
\begin{align*}
\int_{\Omega\setminus \Omega_{\frac{\alpha^{-\sigma}}{2}}}&  \big|\nabla u(x)\big|^2\,dx
\le \int_{\Omega\setminus \Omega_{\frac{\alpha^{-\sigma}}{2}}}  e^{2\theta \alpha \big(d_{\partial\Omega}(x)-\frac{\alpha^{-\sigma}}{2}\big)}\big|\nabla u(x)\big|^2\,dx\\
&\le e^{-\theta \alpha^{1-\sigma}}\int_{\Omega}  e^{2\theta \alpha d_{\partial\Omega}(x)}\big|\nabla u(x)\big|^2\,dx
\stackrel{\eqref{agmon2}}{\le}e^{-\theta \alpha^{1-\sigma}}\,C\alpha^2\|u\|^2_{L^2(\Omega)}.
\end{align*}
Furthermore, due to $|\nabla d_{\partial\Omega}|\le 1$ we have
\begin{align*}
	|\nabla v|^2&=\Big|\alpha^\sigma  \chi'\big(\alpha^\sigma d_{\partial \Omega}(x)\big) u(x) \nabla d_{\partial\Omega}(x)+
	\chi\big(\alpha^\sigma d_{\partial \Omega}(x)\big)\nabla u(x)\Big|^2\\
	&\le 2\Big|\alpha^\sigma  \chi'\big(\alpha^\sigma d_{\partial \Omega}(x)\big) u(x) \nabla d_{\partial\Omega}(x)\Big|^2+
	2\Big|\chi\big(\alpha^\sigma d_{\partial \Omega}(x)\big)\nabla u(x)\Big|^2\le 2\alpha^{2\sigma}\|\chi'\|_\infty^2 \big|u(x)\big|^2
	+2\big|\nabla u(x)\big|^2,
\end{align*}
and with previous estimates we have
\begin{align*}
	\int_{\Omega_{\alpha^{-\sigma}}\setminus \Omega_{\frac{\alpha^{-\sigma}}{2}}} |\nabla v|^2\,dx
	&\ \le 2\alpha^{2\sigma}\|\chi'\|_\infty\int_{\Omega_{\alpha^{-\sigma}}\setminus \Omega_{\frac{\alpha^{-\sigma}}{2}}}\big| u(x)\big|\,dx+2\int_{\Omega_{\alpha^{-\sigma}}\setminus \Omega_{\frac{\alpha^{-\sigma}}{2}}}\Big|\nabla u(x)\Big|^2dx\\
	&\ \le 2\alpha^{2\sigma}\|\chi'\|_\infty\int_{\Omega\setminus \Omega_{\frac{\alpha^{-\sigma}}{2}}}\big| u(x)\big|\,dx+2\int_{\Omega\setminus \Omega_{\frac{\alpha^{-\sigma}}{2}}}\Big|\nabla u(x)\Big|^2dx\\
	&\ \le 2\alpha^{2\sigma}\|\chi'\|_\infty^2Ce^{-\theta \alpha^{1-\sigma}}\,\|u\|^2_{L^2(\Omega)}
	+e^{-\theta \alpha^{1-\sigma}} C\alpha^2\,\|u\|^2_{L^2(\Omega)}\le \alpha^{-N}\|u\|^2_{L^2(\Omega)},
\end{align*}
This results in $\big|q_{\alpha,\alpha^{-\sigma}}(v,v)-q_\alpha(u,u)\big|\le 2 \alpha^{-N}\|u\|^2_{L^2(\Omega)}$.
In addition, for any $u\in U_\alpha$ we have $q_\alpha(u,u)\le E_n(Q_\alpha)\|u\|^2_{L^2(\Omega)}$,
and for the respective $v$ it follows
\[
q_{\alpha,\alpha^{-\sigma}}(v,v)\le q_\alpha(u,u)+2 \alpha^{-N}\|u\|^2_{L^2(\Omega)}\le \big(E_n(Q_\alpha)+2 \alpha^{-N}\big)\|u\|^2_{L^2(\Omega)}
\]

By to the min-max principle we have
\begin{equation}
\begin{aligned}
	E_n(Q_{\alpha,\alpha^{-\sigma}})&\le \max_{v\in J(U_\alpha),\,v\ne0}
	\dfrac{q_{\alpha,\alpha^{-\sigma}}(v,v)}{\|v\|^2_{L^2(\Omega_{\alpha^{-\sigma}})}}
	\le \max_{u\in U_\alpha,\,u\ne0}
	\dfrac{\big(E_n(Q_\alpha)+2 \alpha^{-N}\big)\|u\|^2_{L^2(\Omega)}}{\|v\|^2_{L^2(\Omega_{\alpha^{-\sigma}})}}
\end{aligned}
\label{temp11}
\end{equation}	
Due to \eqref{avn} we have for all large $\alpha$:
\[
1-2\alpha^{-N}\le\dfrac{1}{1+\alpha^{-N}}\le\dfrac{\|u\|^2_{L^2(\Omega)}}{\|v\|_{L^2(\Omega)}}\le \dfrac{1}{1-\alpha^{-N}}\le 1+2\alpha^{-N},
\]
which yields
\begin{align*}
\dfrac{\big(E_n(Q_\alpha)+2 \alpha^{-N}\big)\|u\|^2_{L^2(\Omega)}}{\|v\|^2_{L^2(\Omega_{\alpha^{-\sigma}})}}&\le E_n(Q_\alpha)+2 \alpha^{-N}+ \Big|E_n(Q_\alpha)+2 \alpha^{-N}\Big|2\alpha^{-N}\le E_n(Q_\alpha)+c \alpha^{2-N},
\end{align*}
where we used $E_n(Q_\alpha)=\cO(\alpha^2)$, see Corollary \ref{coro}. The substitution into \eqref{temp11} gives
$E_n(Q_{\alpha,\alpha^{-\sigma}})\le E_n(Q_\alpha)+c \alpha^{2-N}$.
As $N>0$ is arbitrary, the result follows.
\end{proof}

Recall that we have the lower bound $E_n(Q_{\alpha,\alpha^{-\sigma}})\ge E_n(P^-_{\alpha,\alpha^{-\sigma}})$,
see \eqref{qqar3}. We will now proceed with a lower bound for the right-hand side.
A suitable form of the one-dimensional Sobolev inequality will be used, see e.g.
Lemma~8 in~\cite{kuch}:
\begin{lemma}\label{lemsob}
	For any $0<b\le a$ and $f\in H^1(0,a)$ there holds
	\[
	\big|f(0)\big|^2\le b \int_0^a \big|f'(t)\big|^2dt+ \dfrac{2}{b}\int_0^a \big|f(t)\big|^2dt.
	\]
\end{lemma}

First let us make some preliminary steps. Let $\psi$ be an $L^2$-normalized eigenfunction of the one-dimensional operator $T_{\alpha+\frac{k_*}{2},r}$ from Lemma \ref{lem1d}
for its first eigenvalue, then
\[
t_{\alpha+\frac{k_*}{2},r}(\psi,\psi)\equiv 
\int_0^r \big|\psi'(t)\big|^2\,dt-\Big(\alpha+\frac{k_*}{2}\Big) \big|\psi(0)\big|^2=E_1(T_{\alpha+\frac{k_*}{2},r}),
\]
We represent each $v\in \Tilde H^1_0(\Pi_r)$ as
\begin{equation}
v=g\otimes \psi+w,\quad g\otimes \psi:\  (s,t)\to g(s)\psi(t), \label{vwg}
\end{equation}
with $g\in H^1(\TT)$ defined by
\[
g(s):=\int_0^r \overline{\psi(t)}v(s,t)\,dt,\quad s\in\TT.
\]
By construction we have
\[
\int_0^r \overline{\psi(t)}w(\cdot ,t)\, dt=0,
\quad
\int_0^r \overline{\psi(t)}\partial_sw(\cdot ,t)\, dt=\partial_s\int_0^r \overline{\psi(t)}w(\cdot ,t)\, dt=0,
\]
which implies
\begin{align}
	\int_0^r \big|v(s,t)\big|^2\,dt&=\int_0^r \big|g(s)\psi(t)\big|^2\,dt+\int_0^r \big|w(s,t)\big|^2\,dt =\big|g(s)\big|^2+\int_0^r \big|w(s,t)\big|^2\,dt,\ s\in \TT,\nonumber\\
	\text{hence, }\|v\|^2_{L^2(\Pi_r)}&=\|g\|^2_{L^2(\TT)}+\|w\|^2_{L^2(\Pi_r)}.\label{eqsum}
\end{align}

For $\rho\in(0,1)$ and $\alpha>0$ denote by $\Lambda_{\alpha,\rho}$ the Sch\"odinger operator 
in $L^2(\TT)$ given by
\begin{equation}
	\label{ll2}
\Lambda_{\alpha,\rho}: f\mapsto -f''+\bigg[\Big(\alpha(k_*-k)+\frac{k_*^2-2kk_*-k^2}{4}\Big)\one_{(0,\ell)} + \alpha^{2-\rho}\one_{\TT\setminus(0,\ell)}\bigg]f,
\end{equation}
which will play a central role below.

\begin{lemma}\label{propdown}
Let $\sigma,\rho\in(0,1)$ and $\tau\in \big(0,\frac{1}{3}\big)$, and denote
\[
\nu:=\min\{\rho,\tau\}\in(0,1),\quad \mu:=\max\{2-\rho,1+3\tau\}\in(1,2),
\]
then there are $c,c'>0$ and $\alpha_0>0$ such that for any $\alpha>\alpha_0$ and any $v\in \Tilde H^1_0(\Pi_r)$ there holds
\begin{align*}
		p^-_{\alpha,\alpha^{-\sigma}}(v,v)&\ge (-\alpha^2-\alpha k_*-c\alpha^{-\nu})\|g\|^2_{L^2(\TT)} +(1-c'\alpha^{-\nu})\,\lambda_{\alpha,\rho}(g,g)-b\alpha^\mu \big\| w\big\|^2_{L^2(\Pi_r)},
	\end{align*}
where $\lambda_{\alpha,\rho}$ is the sesquilinear form for $\Lambda_{\alpha,\rho}$.
\end{lemma}

\begin{proof}
During the proof all inequalities are considered for $\alpha\to+\infty$, and we set
$r:=\alpha^{-\sigma}$. The spectral theorem for $T_{\alpha+\frac{k_*}{2},r}$
gives for all $s\in\TT$:
\begin{align}
t_{\alpha+\frac{k_*}{2},r}\big(w(s,\cdot),w(s,\cdot)\big)&\ge E_2(T_{\alpha,r}) \big\|w(s,\cdot)\big\|^2_{L^2(0,r)}, \label{i2w}\\
t_{\alpha+\frac{k_*}{2},r}\big(\psi,w(s,\cdot)\big)&=0, \nonumber
\end{align}
which implies
\begin{align*}
	I(v)&:=\int_{\Pi_r} \big|\partial_t v(s,t)\big|^2\,ds\,dt-\Big(\alpha+\frac{k_*}{2}\Big) \int_{\TT}\big|v(s,0)\big|^2\,ds=\int_{\TT} t_{\alpha+\frac{k_*}{2},r}\big( v(s,\cdot),v(s,\cdot)\big)\,ds\\
	&=\int_{\TT} \Big[t_{\alpha+\frac{k_*}{2},r}\big( g(s)\psi,g(s)\psi\big)
	+t_{\alpha+\frac{k_*}{2},r}\big( w(s,\cdot),w(s,\cdot)\big)\Big]\,ds
	=I(g\otimes \psi)+I(w),
\end{align*}
with
\begin{align}
	I(g\otimes\psi)&=\int_{\TT}E_1(T_{\alpha+\frac{k_*}{2},r}) \big|g(s)\big|^2\,ds= E_1(T_{\alpha+\frac{k_*}{2},r}) \|g\|^2_{L^2(\TT)},\nonumber\\
	I(w)&\ge \int_\TT E_2 (T_{\alpha+\frac{k_*}{2},r})\big\|w(s,\cdot)\big\|^2_{L^2(0,r)}\,ds.
\end{align}
The substitution into the expression for $p^-_{\alpha,r}$ gives
\begin{align*}
	p^-_{\alpha,r}(v,v)&=I(v)+J_1(v)+J_2(v),\\
	\text{with } J_1(v)&:=\int_{\Pi_r} \bigg[(1- A r)\big|\partial_s v(s,t)\big|^2-\Big(Ar+\dfrac{k(s)^2}{4}\Big) \big|v(s,t)\big|^2\bigg]\, ds\, dt,\\
	J_2(v)&:=\Big(\alpha+\frac{k_*}{2}\Big)\int_{\TT\setminus(0,\ell)}\big|v(s,0)\big|^2\,ds+\int_0^\ell\dfrac{k_*-k(s)}{2}\,\big|v(s,0)\big|^2\,ds.
\end{align*}
Using the above orthogonality relations we obtain
$J_1(v)=J_1(g\otimes\psi)+J_1(w)$, and one can simplify
\[
J_1(g\otimes\psi)=\int_{\TT} \bigg[(1- A r)\big|g'(s)\big|^2-\Big(Ar+\dfrac{k(s)^2}{4}\Big) \big|g(s)\big|^2\bigg]\, ds.
\]
On the other hand, $J_2(v)=J_2(g\otimes\psi)+J_2(w)+J^\prime(g\otimes\psi,w)$,
\begin{align*}
	J^\prime(g\otimes\psi,w)&:=2\Re\bigg[
	\Big(\alpha+\frac{k_*}{2}\Big)\int_{\TT\setminus(0,\ell)}(g\otimes\psi)(s,0)\overline{w(s,0)}\,ds+
	\int_0^\ell\dfrac{k_*-k(s)}{2}\,(g\otimes\psi)(s,0)\overline{w(s,0)}\,ds	
	\bigg]\\
	&=
	2\Re \bigg[\psi(0)\bigg(
	\Big(\alpha+\frac{k_*}{2}\Big)\int_{\TT\setminus(0,\ell)}g(s)\overline{w(s,0)}\,ds+\int_0^\ell \dfrac{k_*-k(s)}{2}\,g(s)\overline{w(s,0)}\,ds\bigg)\bigg],
\end{align*}
and one can again simplify
\begin{align*}
	J_2(g\otimes\psi)&=\Big(\alpha+\frac{k_*}{2}\Big)\big|\psi(0)\big|^2 \int_{\TT\setminus(0,\ell)} \big|g(s)\big|^2\,ds+\big|\psi(0)\big|^2 \int_0^\ell \dfrac{k_*-k(s)}{2}\,\big|g(s)\big|^2\,ds.
\end{align*}

Overall, we arrive at
\begin{equation}
	\label{pvv}
	p^-_{\alpha,r}(v,v)=p^-_{\alpha,r}(g\otimes\psi,g\otimes\psi)+p^-_{\alpha,r}(w,w)+J'(g\otimes\psi,w),
\end{equation}
with
\begin{align*}
	p^-_{\alpha,r}(g\otimes\psi,g\otimes\psi)&=
	I(g\otimes\psi)+J_1(g\otimes\psi)+J_2(g\otimes\psi)\\
	&=E_1(T_{\alpha+\frac{k_*}{2},r}) \|g\|^2_{L^2(\TT)} 
	+\int_{\TT} \bigg[(1- A r)\big|g'(s)\big|^2-\Big(Ar+\dfrac{k(s)^2}{4}\Big) \big|g(s)\big|^2\bigg]\, ds\\
	&\quad +\Big(\alpha+\frac{k_*}{2}\Big)\big|\psi(0)\big|^2 \int_{\TT\setminus(0,\ell)} \big|g(s)\big|^2\,ds+\big|\psi(0)\big|^2 \int_0^\ell \dfrac{k_*-k(s)}{2}\,\big|g(s)\big|^2\,ds,
\end{align*}
and we recall that $p^-_{\alpha,r}(w,w)=I(w)+J_1(w)+J_2(w)$ and that due to Lemma \ref{lem1d} we have
\begin{align}
E_1(T_{\alpha+\frac{k_*}{2},r})&\ge -\Big(\alpha+\frac{k_*}{2}\Big)^2-r, \label{e1a}\\
E_2(T_{\alpha+\frac{k_*}{2},r})&\ge 0, \label{e2a} \\
\big|\psi(0)\big|^2&\ge 2\Big( \alpha+\frac{k_*}{2}\Big)-r. \label{f0a}
\end{align}

We estimate  $\big|J'(g\otimes\psi,w)\big|\le J'_1+J'_2$ with
\begin{align*}
J'_1&:=  \Big(\alpha+\frac{k_*}{2}\Big)\int_{\TT\setminus(0,\ell)}2\big|\psi(0)\big|\big|g(s)\big|\cdot\big|w(s,0)\big|\,ds,\\
J'_2&:=\int_0^\ell \dfrac{k_*-k(s)}{2}\,2\big|\psi(0)\big|\big|g(s)\big|\cdot\big|w(s,0)\big|\,ds.
\end{align*}
For any $\varepsilon>0$ we have
\begin{equation}
	\label{psi11}
	2\big|\psi(0)\big|\big|g(s)\big|\cdot\big|w(s,0)\big|\le \varepsilon \big|\psi(0)\big|^2\big|g(s)\big|^2+\dfrac{1}{\varepsilon}\big|w(s,0)\big|^2.
\end{equation}
Set $\varepsilon:=1-\alpha^{-\rho}$ with $\rho\in(0,1)$, then $\tfrac{1}{\varepsilon}\le 1+2\alpha^{-\rho}$, and
\begin{equation}
	\label{jpm1}
	\begin{aligned}
	J'_1&\le (1-\alpha^{-\rho})\Big(\alpha+\frac{k_*}{2}\Big)\big|\psi(0)\big|^2\int_{\TT\setminus(0,\ell)} \big|g(s)\big|^2\,ds+(1+2\alpha^{-\rho})\Big(\alpha+\frac{k_*}{2}\Big) \int_{\TT\setminus(0,\ell)} \big|w(s,0)\big|^2\,ds.
	\end{aligned}
\end{equation}

To obtain an upper bound for $J'_2$ we will need a different choice of $\varepsilon$. Remark first that due to \eqref{i2w} and \eqref{e2a} we have
\[
\int_0^r \big|\partial_t w(s,t)\big|^2\,dt-\Big(\alpha+\frac{k_*}{2}\Big)\big|w(\cdot,0)\big|^2\ge 0.
\]
Using $\alpha+\frac{k_*}{2}\ge \frac{\alpha}{2}$ we deduce
\begin{equation}
	\label{sw0}
	\big|w(s,0)\big|^2\le \dfrac{2}{\alpha}\int_{0}^r \big|\partial_t w(s,t)\big|^2\,dt.
\end{equation}
For any $\eta\in (0,1)$ and $b\in(0,r)$ we have, by using \eqref{sw0} and Lemma \ref{lemsob}:
\begin{align*}
	\big|w(s,0)\big|^2&=(1-\eta)\big|w(s,0)\big|^2+\eta\big|w(s,0)\big|^2\\
	&\le \dfrac{2(1-\eta)}{\alpha}\big\|\partial_t w(s,\cdot)\big\|^2_{L^2(0,r)} +\eta\Big( b \big\|\partial_t w(s,\cdot)\big\|^2_{L^2(0,r)} +\dfrac{2}{b}\big\|w(s,\cdot)\big\|^2_{L^2(0,r)}\Big)\\
	&=\dfrac{2}{\alpha}\bigg( 1-\eta\Big[1-\dfrac{\alpha b}{2}\Big]\bigg)\big\|\partial_t w(s,\cdot)\big\|^2_{L^2(0,r)} +\dfrac{2\eta}{b}\big\|w(s,\cdot)\big\|^2_{L^2(0,r)}.
\end{align*}
The choice $b:=\frac{2(1-\theta)}{\alpha}$ with $\theta\in(0,1)$,
simplifies the preceding inequality to
\begin{equation*}
	\big|w(s,0)\big|^2\le \dfrac{2(1-\theta \eta)}{\alpha}\big\|\partial_t w(s,\cdot)\big\|^2_{L^2(0,r)}+\dfrac{\eta\alpha }{1-\theta}\big\|w(s,\cdot)\big\|^2_{L^2(0,r)},
\end{equation*}
and the substitution into \eqref{psi11} gives
\begin{align*}
2\big|\psi(0)\big|&\big|g(s)\big|\cdot\big|w(s,0)\big|\le \varepsilon \big|\psi(0)\big|^2\big|g(s)\big|^2+\dfrac{2(1-\theta \eta)}{\alpha \varepsilon}\big\|\partial_t w(s,\cdot)\big\|^2_{L^2(0,r)}+\dfrac{\eta\alpha }{(1-\theta)\varepsilon}\big\|w(s,\cdot)\big\|^2_{L^2(0,r)}.
\end{align*}

Choose the parameters as
\begin{align*}
\varepsilon&:=\alpha^{-\tau} \text{ with } \tau\in\big(0,\tfrac{1}{3}\big) \qquad \theta:=1-\varepsilon^2\in (0,1),\quad
\eta:= \dfrac{1-\varepsilon}{\theta}=\dfrac{1-\varepsilon}{1-\varepsilon^2}=\dfrac{1}{1+\varepsilon}\in(0,1),
\end{align*}
then
\[
\dfrac{1-\theta \eta}{\varepsilon}=1, \qquad
\dfrac{\eta\alpha }{(1-\theta)\varepsilon}=\dfrac{\alpha}{(1+\varepsilon)\varepsilon^2 \varepsilon}=\dfrac{\alpha^{1+4\tau}}{\alpha^{\tau}+1}\le \alpha^{1+3\tau}.
\]
So we arrive at
\begin{align*}
	2\big|\psi(0)\big|&\big|g(s)\big|\cdot\big|w(s,0)\big|\le \alpha^{-\tau}\big|\psi(0)\big|^2\big|g(s)\big|^2+\dfrac{2}{\alpha}\|\partial_t w(s,\cdot)\big\|^2_{L^2(0,r)}+\alpha^{1+3\tau}\big\|w(s,\cdot)\big\|^2_{L^2(0,r)},
\end{align*}
and the substitution into the expression for $J'_2$ gives
\begin{align*}
	J'_2&\le \alpha^{-\tau}\big|\psi(0)\big|^2\int_0^\ell \dfrac{k_*-k(s)}{2}\,\big|g(s)\big|^2\,ds +\dfrac{2K}{\alpha}\int_0^\ell \|\partial_t w(s,\cdot)\big\|^2_{L^2(0,r)}\,ds
	+K \alpha^{1+3\tau}\int_0^\ell \,\big\|w(s,\cdot)\big\|^2_{L^2(0,r)}\,ds,
\end{align*}
where we denoted
\[
K:=\sup_{s\in (0,\ell)}\dfrac{k_*-k(s)}{2}.
\]

Using the last bound for $J'_2$ and the bound \eqref{jpm1} for $J'_1$ we obtain the
lower bound $J'(g\otimes\psi,w)\ge -J'_1-J'_2$.
Inserting this estimate into the decomposition \eqref{pvv} and collecting separately the terms with $g$ and the terms with $w$ one arrives at
\[
p^-_{\alpha,r}(v,v)\ge Y(g)+Y'(w),
\]
where
\begin{align*}
	Y(g)&:=E_1(T_{\alpha+\frac{k_*}{2},r}) \|g\|^2_{L^2(\TT)} +\int_{\TT} \bigg[(1- A r)\big|g'(s)\big|^2-\Big(Ar+\dfrac{k(s)^2}{4}\Big) \big|g(s)\big|^2\bigg]\, ds\\
	&\quad +\alpha^{-\rho}\Big(\alpha+\frac{k_*}{2}\Big)\big|\psi(0)\big|^2 \int_{\TT\setminus(0,\ell)} \big|g(s)\big|^2\,ds+(1-\alpha^{-\tau})\big|\psi(0)\big|^2 \int_0^\ell \dfrac{k_*-k(s)}{2}\,\big|g(s)\big|^2\,ds,\\
	Y'(w)&:=\int_{\Pi_r} \big|\partial_t w(s,t)\big|^2\,ds\,dt-\Big(\alpha+\frac{k_*}{2}\Big) \int_{\TT}\big|w(s,0)\big|^2\,ds\\
	&\quad + \int_{\Pi_r} \bigg[(1- A r)\big|\partial_s w(s,t)\big|^2-\Big(Ar+\dfrac{k(s)^2}{4}\Big) \big|w(s,t)\big|^2\bigg]\, ds\, dt,\\
 &\quad +\Big(\alpha+\frac{k_*}{2}\Big)\int_{\TT\setminus(0,\ell)}\big|w(s,0)\big|^2\,ds+\int_0^\ell\dfrac{k_*-k(s)}{2}\,\big|w(s,0)\big|^2\,ds\\	
&\quad	-(1+2\alpha^{-\rho})\Big(\alpha+\frac{k_*}{2}\Big) \int_{\TT\setminus(0,\ell)} \big|w(s,0)\big|^2\,ds\\
	&\quad -\dfrac{2K}{\alpha}\int_0^\ell \|\partial_t w(s,\cdot)\big\|^2_{L^2(0,r)}\,ds
	-K\alpha^{1+3\tau}\int_0^\ell \,\big\|w(s,\cdot)\big\|^2_{L^2(0,r)}\,ds.
\end{align*}

Let us first find a lower bound for $Y'(w)$. Choose any $K'>K$ and use
\[
\big|\partial_s w(s,t)\big|^2\ge 0,\quad
\dfrac{k_*-k(s)}{2}\ge 0,\quad
\int_{\TT\setminus(0,\ell)} |f|\, ds\le \int_{\TT} |f|\, ds,
\]
then
\begin{align*}
	Y'(w)&\ge
\int_{\TT} \bigg[\Big(1-\dfrac{2K}{\alpha}\Big)\big\|\partial_t w(s,\cdot)\big\|^2_{L^2(0,r)}-(1+2\alpha^{-\rho})\Big(\alpha+\frac{k_*}{2}\Big) \big|w(s,0)\big|^2\bigg]\,ds
-K'\alpha^{1+3\tau}\,\|w\|^2_{L^2(\Pi_r)}\\
&\ge 
\Big(1-\dfrac{2K}{\alpha}\Big)\int_{\TT} \bigg[\big\|\partial_t w(s,\cdot)\big\|^2_{L^2(0,r)}-(1+3\alpha^{-\rho})\Big(\alpha+\frac{k_*}{2}\Big) \big|w(s,0)\big|^2\bigg]\,ds
-K'\alpha^{1+3\tau}\,\|w\|^2_{L^2(\Pi_r)}.
\end{align*}
We further have
\begin{align*}
\big\|\partial_t w(s,\cdot)\big\|^2_{L^2(0,r)}&-(1+3\alpha^{-\rho})\Big(\alpha+\frac{k_*}{2}\Big) \big|w(s,0)\big|^2\\
&=(1-3\alpha^{-\rho})\Big( \underbrace{\big\|\partial_t w(s,\cdot)\big\|^2_{L^2(0,r)}-\Big(\alpha+\frac{k_*}{2}\Big) \big|w(s,0)\big|^2}_{\ge 0}\Big)\\
&\quad +3\alpha^{-\rho} \Big(\big\|\partial_t w(s,\cdot)\big\|^2_{L^2(0,r)}- 2\Big(\alpha+\frac{k_*}{2}\Big) \big|w(s,0)\big|^2\Big)
\ge 3\alpha^{-\rho} E_1(T_{2\alpha+k^*}) \big\| w(s,\cdot)\big\|^2_{L^2(0,r)},
\end{align*}
and one obtains $Y'(w)\ge \Big[ \Big(1-\dfrac{2K}{\alpha}\Big)3\alpha^{-\rho} E_1(T_{2\alpha+k^*}) -K'\alpha^{1+3\tau}\Big]\, \big\| w\big\|^2_{L^2(\Pi_r)}$.
Using Lemma \ref{lem1d} we estimate $E_1(T_{2\alpha+k^*,r})\ge -( 2\alpha+k_*)^2-1\ge -5\alpha^2$,
and for a suitable $b>0$ we obtain
\[
Y'(w)\ge -b\alpha^\mu \big\| w\big\|^2_{L^2(\Pi_r)},
\quad
\mu:=\max\{2-\rho,1+3\tau\}\in(1,2);
\]
the inclusion for $\mu$ follows from the fact that the above estimates are valid for any
$\rho\in(0,1)$ and $\tau\in\big(0,\frac{1}{3}\big)$.

Now let us proceed with $Y(g)$. Using \eqref{e1a} and \eqref{f0a}, for
$\nu:=\min\{\sigma,\tau\}$ and sufficiently large but fixed $c,c'>0$ we get
\begin{align*}
	Y(g)&\ge -\bigg[\Big( \alpha +\frac{k_*}{2}\Big)^2+\alpha^{-\sigma}\bigg]\|g\|^2_{L^2(\TT)}+\int_{\TT} \bigg[(1- A \alpha^{-\sigma})\big|g'(s)\big|^2-\Big(A\alpha^{-\sigma}+\dfrac{k(s)^2}{4}\Big) \big|g(s)\big|^2\bigg]\, ds\\
&\quad +\alpha^{-\rho}\Big(\alpha+\frac{k_*}{2}\Big)\Big( 2\alpha+k_* -\alpha^{-\sigma}\Big) \int_{\TT\setminus(0,\ell)} \big|g(s)\big|^2\,ds\\
&\quad+(1-\alpha^{-\tau})\Big(2\alpha+k_* -\alpha^{-\sigma}\Big) \int_0^\ell \dfrac{k_*-k(s)}{2}\,\big|g(s)\big|^2\,ds,\\
&\ge (-\alpha^2-\alpha k_* -c\alpha^{-\nu}) \|g\|^2_{L^2(\TT)} +(1-c'\alpha^{-\nu})\bigg[\int_{\TT} \big|g'(s)\big|^2\,ds +\int_0^\ell \Big(\alpha\big(k_*-k(s)\big)\\
&\quad+\frac{k_*^2-2k(s)k_*-k(s)^2}{4}\Big)\big|g(s)\big|^2\Big)\,ds +\alpha^{2-\rho}\int_{\TT\setminus(0,\ell)} \big|g(s)\big|^2\,ds
\bigg]\\
&=(-\alpha^2-\alpha k_* -c\alpha^{-\nu}) \|g\|^2_{L^2(\TT)}+(1-c'\alpha^{-\nu})\lambda_{\alpha,\rho}(g,g).\qedhere
\end{align*}
\end{proof}

We need some a priori estimates for the eigenvalues $E_n(\Lambda'_\alpha)$ and $E_n(\Lambda_{\alpha,\rho})$. Due to the min-max principle we immediately get
\begin{equation}
	\label{lla}
\begin{aligned}
-a\le &E_n(\Lambda_{\alpha,\rho})\le E_n(\Lambda'_\alpha) \text{ for all $\alpha>0$, $n\in\NN$, $\rho\in(0,1)$,}\quad
a:=\sup_{s\in(0,\ell)}\Big|\frac{k_*^2-2k(s)k_*-k(s)^2}{4}\Big|.
\end{aligned}
\end{equation}
\begin{lemma}\label{apriori}
Let $s_*\in [0,\ell]$ be a maximum of $k$ on $[0,\ell]$. If for some
$m>0$ one has $k(s)-k(s_*)=\cO\big(|s-s_*|^m\big)$ for $s\to s_*$,
then for each $n\in\NN$ and $\rho\in (0,1)$ there holds
	\[
	E_n(\Lambda'_\alpha)=\cO(\alpha^{\frac{2}{2+m}}) \text{ and }
	E_n(\Lambda_{\alpha,\rho})=\cO(\alpha^{\frac{2}{2+m}}) \text{ as }
	\alpha\to+\infty.
	\]
In particular, without any additional assumption one has
	\[
E_n(\Lambda'_\alpha)=\cO(\alpha^{\frac{2}{3}}) \text{ and }
E_n(\Lambda_{\alpha,\rho})=\cO(\alpha^{\frac{2}{3}}) \text{ as }
\alpha\to+\infty.
\]	
\end{lemma}

\begin{proof}
In view of \eqref{lla} it is sufficient to obtain an upper bound for $E_n(\Lambda'_\alpha)$.	
By assumption one can find some $M>0$ and $\delta_0\in (0,\ell)$ with
	\[
	k_*-k(s)=k(s_*)-k(s)\le M |s_*-s|^m \text{ for all $s\in\TT$ with $|s-s_*|<\delta_0$.}
	\]
The length of the intervall $I:=(s_*-\delta_0,s_*+\delta_0)\cap(0,\ell)$
is at least $\delta_0$, so for any $\delta\in(0,\delta_0)$ we can choose a subinterval $I_\delta\subset I$ of length $\delta$. By the min-max principle
one has $E_n(\Lambda'_\alpha)\le E_n(\Lambda'_{\alpha,\delta})$,
where $\Lambda'_{\alpha,\delta}$ is the operator
\[
	\Lambda'_{\alpha,\delta}:\ f\mapsto -f''+\Big(\alpha (k_*-k)+\frac{k_*^2-2kk_*-k^2}{4}\Big)f
\]
in $L^2(I_\delta)$ with Dirichlet boundary conditions. For any $s\in I_\delta$ we have
\[
\alpha \big(k_*-k(s)\big)+\frac{k_*^2-2k(s)k_*-k(s)^2}{4}\le \alpha M\delta^m+a,
\]
which gives $\Lambda'_{\alpha,\delta}\le \pi^2 n^2\delta^{-2}+\alpha M\delta^m+a$,
and by choosing $\delta:=\alpha^{-\frac{1}{2+m}}$ we obtain the first claim.

The function $k$ is $C^2$ and, therefore, Lipschitz, so the assumption is always satisfied for $m=1$, which gives the last claim.
\end{proof}

\begin{proposition}\label{prop36}
Let $n\in\NN$, $\rho\in(0,1)$ and $\tau\in(0,\frac{1}{3})$. Then there are $c,c'>0$ and $\alpha_0>0$ such that for any $\alpha>\alpha_0$ there holds
\begin{align*}
E_n(Q_\alpha)&\ge -\alpha^2-k_*\alpha +(1-c'\alpha^{-\tau})E_n(\Lambda_{\alpha,\rho})-c\alpha^{-\tau}.
\end{align*}
\end{proposition}

\begin{proof}
	Let $\rho\in (0,1)$. Denote $\mu:=\max\{2-\rho,1+3\mu\}$ and $\nu:=\min\{\sigma,\tau\}$. Let $\cG\subset L^2(\Pi_r)$  be the closure of the subspace of all $w$ in \eqref{vwg} with $v\in \Tilde H^1_0(\Pi_r)$, and let $I:\cG\to\cG$ be the identity map.
	The decomposition \eqref{vwg} together with Proposition \ref{propdown} show $E_n(P^-_{\alpha,\alpha^{-\sigma}})\ge E_n(L_\alpha\oplus -b\alpha^\mu I)$ with
	\begin{gather*}
		L_\alpha:=-\alpha^2-\alpha k_* +(1-c'\alpha^{-\nu})\,\Lambda_{\alpha,\rho}
		-c\alpha^{-\nu}.
	\end{gather*}
By Lemma \ref{apriori}	we have $E_n(L_\alpha)=-\alpha^2+\cO(\alpha)<-b\alpha^\mu$,
so
\begin{gather}
E_n(L_\alpha\oplus -b\alpha^\mu I)
	\equiv \min\big\{E_n(L_\alpha),-b\alpha^\mu\big\}=E_n(L_\alpha),\nonumber\\
E_n(P^-_{\alpha,\alpha^{-\sigma}})\ge E_n(L_\alpha).\label{low3}
\end{gather}
In addition, we can choose $\sigma>\tau$ to have $\nu=\tau$. Using Corollary \ref{cor4} we obtain for any $M>0$:
\begin{align*}
E_n(Q_\alpha)&=E_n(Q_{\alpha,\alpha^{-\sigma}})+\cO(\alpha^{-M})\stackrel{\eqref{qqar2}}{=}E_n(P_{\alpha,\alpha^{-\sigma}})+\cO(\alpha^{-M})\\
&\stackrel{\eqref{qqar3}}{\ge}E_n(P^-_{\alpha,\alpha^{-\sigma}})+\cO(\alpha^{-M})
\stackrel{\eqref{low3}}{\ge}E_n(L_\alpha)+\cO(\alpha^{-M}),
\end{align*}
and we obtain the claim by taking $M>\tau$.
\end{proof}

We summarize Propositions \ref{propup} and \ref{prop36} as follows:
\begin{corollary}\label{cor}
Let $n\in\NN$, $\rho,\sigma\in(0,1)$ and $\tau\in(0,\frac{1}{3})$. Then there are $c,c',b,b'>0$ such that for $\alpha\to+\infty$ one has
\begin{align*}
(1-c'\alpha^{-\tau})E_n(\Lambda_{\alpha,\rho})-c\alpha^{-\tau}&\le E_n(Q_\alpha) +\alpha^2 +k_*\alpha\le 	(1+b\alpha^{-\sigma})E_n(\Lambda'_\alpha) + b'\alpha^{-\sigma}.
\end{align*}
\end{corollary}	

\section{Reduction to an effective operator on $(0,\ell)$}\label{sec4}

Now we are going to show that $E_n(\Lambda_{\alpha,\rho})$ and $E_n(\Lambda'_\alpha)$ are asymptotically close to each other. This will allow
to transform the two-sided estimate of Corollary \ref{cor} into an asymptotic expansion.  Our analysis will be based on the following result, see e.g. \cite[Lemma 2.1]{Ex03}:

\begin{proposition} \label{Comp}
	Let	$T$ be a non-negative self-adjont operator with compact resolvent in an infinite-dimensional Hilbert space $\cH$, defined by a closed hermitian sesquilinear form $t$, and $T'$ be a lower semibounded
	self-adjoint operator with compact resolvent in an infinite-dimensional Hilbert space $\cH'$, defined by a closed hermitian sesquilinear form $t'$.
	Assume that there are a linear map $J: \cD(t)\to \cD(t')$
	and  $\delta_1,\delta_2\in [0,+\infty)$ such that
	for all $g\in\cD(t)$ there holds
	\begin{align*}
		\|g \|_\cH^2  - \|Jg\|_{\cH'}^2&\le\delta_1 \Big(t(g,g) +  \|g \|_\cH^2\Big), \\
		t'(Jg,Jg) - t(g,g) &\le \delta_2 \Big( t(g,g) +  \|g \|_\cH^2\Big).
	\end{align*}
	Then for any $n\in\NN$ with 
	\begin{equation}
		\label{condn}
		\delta_1 \big(E_n(T)+1\big)<1
	\end{equation}
	one has
	\[
	E_n(T') \leq E_n(T) + 
	\frac{\Big(\delta_1 E_n(T)+\delta_2\Big)\Big( E_n(T)+ 1\Big)}{1-\delta_1\big( E_n(T)+ 1\big)}.
	\]
\end{proposition}

\begin{lemma}\label{lem11}
	For any fixed $n\in\NN$ and $\rho\in(0,\frac{1}{3})$ one has for $\alpha\to+\infty$:
	\begin{gather*}
		E_n(\Lambda_{\alpha,\rho})=E_n(\Lambda'_\alpha)+R_\alpha,\\
R_\alpha=\begin{cases}\cO(\alpha^{-\frac{1-\rho}{4}}), &\text{if }E_n(\Lambda'_\alpha)=\cO(1),\\
	\cO\Big( \big(\alpha^{-\frac{3-\rho}{4}} E_n(\Lambda'_\alpha)+\alpha^{-\frac{1-\rho}{4}}\big)E_n(\Lambda'_\alpha)\Big), & \text{if } E_n(\Lambda'_{\alpha})\to +\infty.
	\end{cases}
	\end{gather*}
\end{lemma}	

\begin{proof}
Remark first that $\Lambda'_\alpha$ is monotonically increasing with respect to $\alpha$, so each eigenvalue
is a monotonically increasing function of $\alpha$ and either $E_n(T'_\alpha)=\cO(1)$ or $E_n(\Lambda'_{\alpha})\to +\infty$,
so the claim covers all possible cases.

In view of \eqref{lla} it is sufficient to find an upper bound of for $E_n(\Lambda'_{\alpha})$
in terms of $E_n(\Lambda_{\alpha,\rho})$. We will apply Proposition \ref{Comp} to the operators
\begin{equation}
T:=\Lambda_{\alpha,\rho}+a,
\qquad
T':=\Lambda'_{\alpha}+a
\label{ttaa}
\end{equation}
 with $a$ from \eqref{lla}. For convenience we denote
\begin{align*}
U&:=k_*-k, & V&:=\frac{k_*^2-2kk_*-k^2}{4},
\end{align*}
and recall that the sesquilinear form $t$ for $T$ and the sesquilinear form $t'$ for $T'$ are
given by
\begin{align*}
	t(g,g)&=\int_0^{\ell}
	\Big(
	|g'|^2	+(\alpha U+V+a)|g|^2\Big)\,ds+
	\int_{\TT\setminus(0,\ell)}
	\Big(
	\big|g'|^2+(\alpha^{2-\rho}+a)|g|^2\Big)\,ds \text{ with }\cD(t)=H^1(\TT),\\
	t'(g,g)&=\int_0^{\ell}
	\Big(|g'|^2	+(\alpha U+V+a)|g|^2\Big)\,ds\quad\text{with } \cD(t')=H^1_0(0,\ell).	
\end{align*}

Pick any $h\in C^\infty_c(\RR)$ with $\supp h\in(-\frac{\ell}{2},\frac{\ell}{2})$ and $h(0)=1$ and define
a linear map $J:\ H^1(\TT)\to H^1_0(0,\ell)$ by
\[
(Jg)(s):=g(s)-g(0)h(\alpha^q s)-g(\ell)h\big(\alpha^q(s-\ell)\big),\quad s\in(0,\ell),
\]
with a parameter $q>0$ whose value will be chosen later.
Due to Lemma \ref{lemsob} we have
\begin{equation}
	\label{sob2}
\begin{aligned}	
\big|g(0)\big|^2&+\big|g(\ell)\big|^2\le 2\delta\int_{\TT\setminus(0,\ell)}\big|g'(s)\big|^2\,ds+
4\delta^{-1}\int_{\TT\setminus(0,\ell)}\big|g(s)\big|^2\,ds\\
&\quad
\text{ for any $g\in H^1\big(\TT\setminus(0,\ell)\big)$ and any $0<\delta<L-\ell$.}
\end{aligned}
\end{equation}

Using the inequality $|a+b|^2\ge (1-\varepsilon)|a|^2-\varepsilon^{-1}|b|^2$ (valid for any $a,b\in\CC$ and $\varepsilon>0$)
we estimate for any $g\in H^1(\TT_\alpha)$:
\begin{align*}
	\|Jg&\|^2_{L^2(0,\ell_\alpha)}=\int_0^{\ell}\Big|g(s)-\big[g(0)h(\alpha^q s)+g(\ell)h\big(\alpha^q(s-\ell)\big)\big]\Big|^2\,ds\\
	&\ge (1-\varepsilon)\int_0^{\ell}\big|g(s)\big|^2\, ds
	-\varepsilon^{-1}\underbrace{\int_0^{\ell}\Big|g(0)h(\alpha^q s)+g(\ell)h\big(\alpha^q(s-\ell)\big)\Big|^2\,ds}_{=:I},
\end{align*}
therefore,
\begin{equation}
	\label{ggg1}
	\|g\|^2_{L^2(\TT)}-\|Jg\|^2_{L^2(0,\ell)}\le
	\int_{\TT\setminus(0,\ell)}\big|g(s)\big|^2\,ds
	+\varepsilon\int_0^{\ell}\Big|g(s)\big|^2\,ds+\varepsilon^{-1} I.
\end{equation}
Using $|a+b|^2\le 2|a|^2+2|b|^2$ we obtain
\begin{align*}
I&\le 2\int_0^{\ell}\big|h(\alpha^q s)\big|^2\,ds\cdot \big|g(0)\big|^2+
	2\int_0^{\ell}\Big|h\big(\alpha^q(s-\ell)\big)\Big|^2\,ds\cdot \big|g(\ell)\big|^2.
\end{align*}
Using the substitution $s:=\alpha^{-q}t$ we estimate
\[
\int_0^{\ell}\big|h(\alpha^q s)\big|^2\,ds=\alpha^{-q} \int_0^{\ell}\big|h(t)\big|^2\,dt=\cO(\alpha^{-q})
\]
and analogously
\[
\int_0^{\ell}\Big|h\big(\alpha^q(s-\ell)\big)\Big|^2\,ds=\cO(\alpha^{-q}),
\]
which yields (with suitable $\alpha$-independent $c_j>0$)
\begin{align*}
I&\le c_1\alpha^{-q}\Big( \big|g(0)\big|^2+\big|g(\ell)\big|^2\Big)
\stackrel{\eqref{sob2}}{\le} c_2\alpha^{-q}\bigg(\delta\int_{\TT\setminus(0,\ell)}\big|g'(s)\big|^2\,ds+
\delta^{-1}\int_{\TT\setminus(0,\ell)}\big|g(s)\big|^2\,ds\bigg).
\end{align*}
The substitution into \eqref{ggg1} gives
\begin{align*}
	\|g\|^2_{L^2(\TT)}-\|Jg\|^2_{L^2(0,\ell)}&\le \varepsilon\int_0^{\ell}\big|g(s)\big|^2\,ds +c_2\varepsilon^{-1}\delta \alpha^{-q}\int_{\TT\setminus(0,\ell)}\big|g'(s)\big|^2\,ds\\
	&\quad+(1+c_2\varepsilon^{-1}\delta^{-1} \alpha^{-q})	\int_{\TT\setminus(0,\ell)}\big|g(s)\big|^2\,ds
\end{align*}
Each of the first two  integrals on the right-hand side is bounded from above by $t(g,g)+\|g\|^2_{L^2(\TT)}$, while for the last integral we have
\[\int_{\TT\setminus(0,\ell)}\big|g(s)\big|^2\,ds
\le \alpha^{\rho-2}t(g,g).
\]
This gives
$\|g\|^2_{L^2(\TT)}-\|Jg\|^2_{L^2(0,\ell)}\le c_3(\varepsilon+\varepsilon^{-1}\delta\alpha^{-q} +\alpha^{\rho-2} +\varepsilon^{-1}\delta^{-1} \alpha^{-q}\alpha^{\rho-2})\big(t(g,g)+\|g\|^2_{L^2(\TT)}\big)$.
We optimize the right-hand side by taking
$\varepsilon=\varepsilon^{-1}\delta\alpha^{-q}=\varepsilon^{-1}\delta^{-1} \alpha^{-q}\alpha^{\rho-2}$,
i.e. $\delta:=\alpha^{\frac{\rho-2}{2}}$, $\varepsilon:=\alpha^{\frac{\rho-2q-2}{4}}$,
and arrive at
\begin{equation}
		\label{est1}
		\|g\|^2_{L^2(\TT)}-\|Jg\|^2_{L^2(0,\ell)}\le c_4
		 \alpha^{\frac{\rho-2q-2}{4}}\Big( t(g,g)+\|g\|^2_{L^2(\TT)}\Big).
\end{equation}	

With the help of the inequality
$|a+b|^2\le (1+\varepsilon)|a|^2+2\varepsilon^{-1}|b|^2$ for any $a,b\in\CC,\ \varepsilon\in(0,1)$,
we obtain
\begin{align*}
	t'&(Jg,Jg)=\int_0^{\ell}\Big|g'(s)-\alpha^{q}\big[g(0)h'(\alpha^q s)+g(\ell)h'\big(\alpha^q(s-\ell)\big)\big]\Big|^2ds\\
	&\quad \qquad+\int_0^{\ell}\Big(\alpha U(s)+V(s)+a\Big)\Big|g(s)-\big[g(0)h(\alpha^q s)+g\big(\alpha^q(s-\ell)\big)\big]\Big|^2ds\\
	&\le (1+\varepsilon)\int_0^{\ell}\big|g'(s)\big|^2\, 
	ds+2\varepsilon^{-1}\alpha^{2q}\int_0^{\ell}\Big|g(0)h'(\alpha^q s)+g(\ell)h'\big(\alpha^q(s-\ell)\big)\Big|^2ds\\
	&\quad+ (1+\varepsilon)\int_0^{\ell}\Big(\alpha U(s)+V(s)+a\Big)\big|g(s)\big|^2 ds\\
	&\quad+2\varepsilon^{-1}\int_0^{\ell}\Big(\alpha U(s)+V(s)+a\Big)\Big|g(0)h(\alpha^q s)+g(\ell)h\big(\alpha^q(s-\ell)\big)\Big|^2ds,
\end{align*}
and then
\begin{equation}
	\label{tpm}
\begin{aligned}
	t'(Jf,Jf)-t(f,f)&\le \varepsilon\int_0^{\ell}\Big(|g'|^2 +(\alpha U+V+a)|g|^2\Big)\,ds\\
	&\quad +2\varepsilon^{-1}\alpha^{2q}\underbrace{\int_0^{\ell}\Big[g(0)h'(\alpha^q s)+g(\ell)h'\big(\alpha^q(s-\ell)\big)\Big|^2ds}_{=:I_1}\\
	&\quad +2\varepsilon^{-1}\underbrace{\int_0^{\ell}\big(\alpha U(s)+V(s)+a\big)\Big|g(0)h(\alpha^q s)+g(\ell)h\big(\alpha^q(s-\ell)\big)\Big|^2ds}_{=:I_2}.
\end{aligned}
\end{equation}
We have
\begin{align*}
	I_1\le 2\int_0^{\ell}\big|h'(\alpha^q s)\big|^2\,ds&\cdot \big|g(0)\big|^2
	+2\int_0^{\ell}\Big|h'\big(\alpha^q(s-\ell)\big)\Big|^2\,ds\cdot \big|g(\ell)\big|^2,\\
	\int_0^{\ell}\big|h'(\alpha^q s)\big|^2\,ds&\stackrel{s=\alpha^{-q}t}{=}\alpha^{-q} \int_0^{\ell}\big|h'(t)\big|^2\,dt=\cO(\alpha^{-q}),\\
\int_0^{\ell}\Big|h'\big(\alpha^q(s-\ell)\big)\Big|^2\,ds
&\stackrel{s=\ell+\alpha^{-q}t}{=}
\int_{-\ell}^{0}\Big|h'\big(\alpha^q(t)\big)\Big|^2\,ds=\cO(\alpha^{-q}),
\end{align*}
hence, $I_1\le c_5\alpha^{-q}\Big(\big|g(0)\big|^2+\big|g(\ell)\big|^2\Big)$. Similarly,
\begin{align*}
	I_2&\le c_4\alpha\bigg(\int_0^{\ell_\alpha}\big|h(\alpha^q s)\big|^2\,ds\cdot \big|g(0)\big|^2
	+\int_0^{\ell_\alpha}\Big|h\big(\alpha^q(s-\ell)\big)\Big|^2\,ds\cdot \big|g(\ell)\big|^2\bigg)\le c_6\alpha^{1-q}\Big(\big|g(0)\big|^2+\big|g(\ell)\big|^2\Big).
\end{align*}
The substitution into \eqref{tpm} yields
\begin{align*}
		t'(Jg,Jg)-t(g,g)&\le \varepsilon\int_0^{\ell}\Big(|g'|^2 +(\alpha U+V+a)|g|^2\Big)\,ds +c_7\varepsilon^{-1}(\alpha^{q}+\alpha^{1-q})\Big(\big|g(0)\big|^2+\big|g(\ell)\big|^2\Big).
\end{align*}
The optimization with respect $q$ implies the choice $q:=\frac{1}{2}$, and
\begin{align}
t'(Jg,Jg)-t(g,g)&\le \varepsilon t(g,g) +c_8\varepsilon^{-1}\alpha^{\frac{1}{2}}\Big(\big|g(0)\big|^2+\big|g(\ell)\big|^2\Big).
\label{tpm2}
\end{align}
Using \eqref{sob2} we estimate
\begin{align*}
	\big|g(0)\big|^2+\big|g(\ell)\big|^2&\le  2\delta\int_{\TT\setminus(0,\ell)}\big|g'(s)\big|^2\,ds+
	4\delta^{-1}\int_{\TT\setminus(0,\ell)}\big|g(s)\big|^2\,ds\\
	&\le 2\delta\int_{\TT\setminus(0,\ell)}\big|g'(s)\big|^2\,ds+
	4\delta^{-1}\alpha^{\rho-2}\cdot\alpha^{2-\rho}\int_{\TT\setminus(0,\ell)}\big|g(s)\big|^2\,ds\le c_9(\delta+\delta^{-1}\alpha^{\rho-2})t(g,g),
\end{align*}
and by \eqref{tpm2} we obtain
\[
t'(Jg,Jg)-t(g,g)\le c_{10}(\varepsilon +\varepsilon^{-1}\alpha^{\frac{1}{2}}\delta +
\varepsilon^{-1}\alpha^{\frac{1}{2}}\delta^{-1}\alpha^{\rho-2})\big(t(g,g)+\|g\|^2_{L^2(\TT)}\big).
\]
We optimize the right-hand side by choosing
$\varepsilon=\varepsilon^{-1}\alpha^{\frac{1}{2}}\delta=\varepsilon^{-1}\alpha^{\frac{1}{2}}\delta^{-1}\alpha^{\rho-2}$,
i.e. $\varepsilon:=\alpha^{\frac{\rho-1}{4}},\quad
\delta:=\alpha^{\frac{\rho-2}{2}}$,
which yields
\begin{equation}
	\label{est2}
	t'(Jg,Jg)-t(g,g)\le c_{11} \alpha^{\frac{\rho-1}{4}}\Big( t(g,g)+\|g\|^2_{L^2(\TT)}\Big).
\end{equation}	
The estimate \eqref{est1} with the chosen value $q=\frac{1}{2}$ and \eqref{est2} show that we are
in the situation of Proposition \ref{Comp} with
\[
\delta_1:=c_4\alpha^{\frac{\rho -3}{4}},\quad
\delta_2:=c_{11}\alpha^{\frac{\rho -1}{4}}.
\]
Let $n\in \NN$ be fixed. From now on assume $\rho\in(0,\tfrac{1}{3})$, then by Lemma \ref{apriori} we have
\begin{equation}
	\label{lokal22}
\delta_1 \big(E_n(T)+1\big)=\cO(\alpha^{\frac{\rho -3}{4}})\cO(\alpha^{\frac{2}{3}})=\cO(\alpha^{\frac{3\rho-1}{12}})=o(1),
\end{equation}
and the assumption \eqref{condn} of Proposition \ref{Comp} is satisfied (for all sufficiently large $\alpha$). Using the definitions \eqref{ttaa} of $T$ and $T'$ we obtain the inequality
$E_n(\Lambda'_\alpha)\le E_n(\Lambda_{\alpha,\rho})+R_\alpha$ with
\begin{align*}
R_\alpha&:=\frac{\big(\delta_1 E_n(T)+\delta_2\big)\big( E_n(T)+ 1\big)}{1-\delta_1\big( E_n(T)+ 1\big)}=\cO\Big( \delta_1 \big(E_n(\Lambda_{\alpha,\rho})+a\big)+\delta_2\big)\big( E_n(\Lambda_{\alpha,\rho})+a+ 1\big)\Big)
\end{align*}

If $E_n(\Lambda'_\alpha)=\cO(1)$, then also $E_n(\Lambda_{\alpha,\rho})=\cO(1)$ due to \eqref{lla}, so
\[
\cO\Big( \delta_1 \big(E_n(\Lambda_{\alpha,\rho})+a\big)+\delta_2\big)=\cO(\delta_1+\delta_2)=\cO(\alpha^{\frac{\rho -3}{4}}+\alpha^{\frac{\rho- 1}{4}})=\cO(\alpha^{\frac{\rho- 1}{4}}),
\]
and finally $R_\alpha=\cO(\alpha^{\frac{\rho -1}{4}})$.

Let $E_n(\Lambda'_\alpha)\to+\infty$. Due to \eqref{lla} we have $E_n(\Lambda_{\alpha,\rho})=\cO\big(E_n(\Lambda'_\alpha)\big)$
and then $E_n(\Lambda_{\alpha,\rho}+a+ 1\big)=\cO\big( E_n(\Lambda_{\alpha,\rho}\big)=\cO\big(E_n(\Lambda'_\alpha)\big)$.
The substitition of these estimates and \eqref{lokal22} into the expression for $R_\alpha$ gives the result.
\end{proof}

\begin{theorem}[Effective operator]\label{thmeff}
Let $n\in\NN$ be fixed and $\vartheta\in(\frac{1}{6},\frac{1}{4})$.
\begin{enumerate}
	\item[(i)] If $E_n(\Lambda'_\alpha)=\cO(1)$ for $\alpha\to+\infty$, then
	$E_n(Q_\alpha)=-\alpha-k_*\alpha +E_n(\Lambda'_\alpha)+\cO(\alpha^{-\vartheta})$ for $\alpha\to+\infty$.
	\item[(ii)] If $E_n(\Lambda'_\alpha)\to+\infty$  for $\alpha\to+\infty$, then for $\alpha\to+\infty$ one has
	\[
	E_n(Q_\alpha)=-\alpha^2-k_*\alpha+E_n(\Lambda'_\alpha)+\cO\Big( \alpha^{-\vartheta}\big(
	\alpha^{-\frac{1}{2}}E_n(\Lambda'_\alpha)+1\big)E_n(\Lambda'_\alpha)\Big).
	\]
\end{enumerate}	
\end{theorem}

\begin{proof}
(i) The substitution of the result of Lemma \ref{lem11} into Corollary \ref{cor} shows that for any $\rho\in(0,\frac{1}{3})$ and $\sigma,\tau\in(0,\frac{1}{3})$ one has
	$E_n(Q_\alpha)=-\alpha^2 -\alpha k_*=E_n(\Lambda'_\alpha) + \cO(\alpha^{-\sigma}+\alpha^{-\frac{1-\rho}{4}}+\alpha^{-\tau})$.
	Denoting $\vartheta:=\frac{1-\rho}{4}\in (\frac{1}{6},\frac{1}{4})$ and taking any
	 $\sigma\in(\vartheta,1)$ and $\tau\in(\vartheta,\frac{1}{3})$ we arrive at the sought conclusion.
	
(ii) One proceeds in the same way: The result of Lemma \ref{lem11} is substituted into Corollary \ref{cor}, which
gives, with $\vartheta:=\frac{1-\rho}{4}\in (\frac{1}{6},\frac{1}{4})$ and with any $\sigma\in(0,1)$ and $\tau\in(0,\frac{1}{3})$,
\begin{align*}
E_n(Q_\alpha)&=-\alpha^2-k_*\alpha=E_n(\Lambda'_\alpha)
+
\cO\Big( \big[\alpha^{-\vartheta}\big(
\alpha^{-\frac{1}{2}}E_n(\Lambda'_\alpha)+1\big) +\alpha^{-\tau}+\alpha^{-\sigma}\big]E_n(\Lambda'_\alpha)+\alpha^{-\tau}+\alpha^{-\sigma}\Big).
\end{align*}
and we obtain the sought result by taking any $\sigma\in(\vartheta,1)$ and $\tau\in(\vartheta,\frac{1}{3})$.
\end{proof}

\section{Main results}\label{sec5}

Now we apply Theorem \ref{thmeff} to several specific situations. The first one is very straightforward:
\begin{theorem}[Constant curvature]\label{kconst}
Assume that the curvature	$k$ is constant on $(0,\ell)$, i.e. $k\equiv k_*$, then for any fixed $n\in\NN$ and $\vartheta\in(\frac{1}{6},\frac{1}{4})$ there holds
\[
E_n(Q_\alpha)=-\alpha^2-k_*\alpha-\tfrac{k_*^2}{2}+\tfrac{\pi^2}{n^2\ell^2}+\cO(\alpha^{-\vartheta}) \text{ for }\alpha\to +\infty.
\]
\end{theorem}

\begin{proof}
	For $k\equiv k_*$ on $(0,\ell)$ the operator $\Lambda'_\alpha$ is independent of $\alpha$: there holds $\Lambda'_\alpha=D-\frac{k_*^2}{2}$ with $D$:=the Dirichlet Laplacian on $(0,\ell)$,
	so we are in the situation of Theorem~\ref{thmeff}, while
	\[
	E_n(\Lambda'_\alpha)=E_n(D)-\dfrac{k_*^2}{2}=\dfrac{\pi^2n^2}{\ell^2}-\dfrac{k_*^2}{2}.\qedhere
	\]
\end{proof}

%

Another case which can be directly deduced from the existing results is as follows:
\begin{theorem}[Maximum curvature attained inside $\Gamma$]\label{kmax0}
	Let $m\in 2\NN$ and the boundary $\partial \Omega$ be $C^{m+3}$-smooth.
	Assume that $k$ assumes its strict maximum on $[0,\ell]$ at some point $s_*\in(0,\ell)$ such that 	$k^{(m)}(s_*)<0$ and $k^{(j)}(s_*)=0$ for all $j\in\{1,\dots,m-1\}$. Then for and fixed $n\in\NN$ and $\alpha\to+\infty$ there holds
	\[
	E_n(Q_\alpha)=-\alpha^2-k_*\alpha+ \Big(-\tfrac{k^{(m)}(0)}{m!}\Big)^\frac{2}{m+2} E_n(Z_{m})\alpha^{\frac{2}{m+2}}+\cO(\alpha^{\frac{1}{m+2}+\varepsilon}),
	\]
	with any $\varepsilon>0$ for $m=2$ and $\varepsilon=0$ for $m\ge 4$, 	where $Z_m$ is the Schrödinger operator in $L^2(\RR)$ given by
	\[
	(Z_m f)(t)=-f''(t)+t^m f(t).
	\]
	In particular, for $m=2$ one has for any $\varepsilon>0$
	\begin{equation}
		\label{mm2a}
		E_n(Q_\alpha)=-\alpha^2-k_*\alpha +(2n-1)\sqrt{-\tfrac{k''(0)}{2}} \cdot \sqrt{\alpha}+\cO(\alpha^{\frac{1}{4}+\varepsilon}).
	\end{equation}
\end{theorem}

\begin{proof}
Denote $U:=k_*-k$, and let $D$ be the Dirichlet Laplacian in $L^2(0,\ell)$. The analysis of $D+\alpha U$ is covered by the classical results of semiclassical analysis \cite{hbook,hs1,mr}, in particular, for any fixed $n\in\NN$ we have
\[
E_n(D+\alpha U)=\Big(\dfrac{U^{(m)}(0)}{m!}\Big)^\frac{2}{m+2} E_n(Z_{m})\alpha^{\frac{2}{m+2}}+\cO(\alpha^{\frac{1}{m+2}}),
\]
see \cite[Theorem 2.1]{mr}, and then $E_n(\Lambda'_\alpha)=E_n(D+\alpha U)+\cO(1)$. The substitution into Theorem \ref{thmeff}(ii) gives
\[
E_n(Q_\alpha)=-\alpha^2-k_*\alpha +\Big(\dfrac{U^{(m)}(0)}{m!}\Big)^\frac{2}{m+2} E_n(Z_{m})\alpha^{\frac{2}{m+2}}
+\cO(\alpha^{\frac{1}{m+2}}+\alpha^{\frac{2}{m+2}-\vartheta})
\]
with any $\vartheta\in(\frac{1}{6},\frac{1}{4})$, which gives the claim (remark that for $m\ge 4$ it is possible to choose $\vartheta\ge\frac{1}{m+2}$).
The formula \eqref{mm2a} follows by the using the well-known formulas for the eigenvalues of $Z_2$ (one-dimensional harmonic oscillator),
	$E_n(Z_2)=2n-1$.
\end{proof}

In order to cover several cases of variable curvature with a minimum attained at an end point of $\Gamma$, for $m\in\NN$ and $\beta>0$ we denote by
$S_{m,\beta}$ the Schr\"odinger operator in $L^2(0,+\infty)$ given by
\begin{equation}
	\label{smb}
(S_{m,\beta} f)(t)=-f''(t)+\beta t^m f(t)
\end{equation}
with Dirichlet boundary condition. It is standard to see that $S_{m,\beta}$ has compact resolvent and all its eigenvalues are simple.
In addition, the homogeneity of the potential implies
\begin{equation}
	\label{scale}
	E_n(S_{m,\beta})=\beta^{\frac{2}{2+m}}E_n(S_{m,1}) \text{ for all $m,n\in\NN$ and $\beta>0$.}
\end{equation}

The following result is a straightforward adaptation of the known results of semiclassical results to the case of a potential attaining its minimum at a boundary point.
\begin{lemma}\label{semikl}
	Let $a>0$ and $m\in\NN$. Let $U\in C^{m+1}([0,a])$ with $U(0)=0$ such that
	\begin{itemize}
		\item[(i)] the point $0$ is a strict minimum of $U$ on $[0,a]$, i.e. $U(t)>0$ for all $t\in(0,a]$,
		\item[(ii)] $U^{(m)}(0)> 0$ and $U^{(j)}(0)=0$ for all $j\in\{1,\dots,m-1\}$.
	\end{itemize}
	Let $D$ be the Dirichlet Laplacian in $L^2(0,a)$, then for $\alpha\to+\infty$ one has
	\[
	E_n(D+\alpha U)=\Big(\dfrac{U^{(m)}(0)}{m!}\Big)^{\frac{2}{2+m}} E_n(S_{m,1})\,\alpha^{\frac{2}{2+m}}+\cO(\alpha^{\frac{2}{3+m}}).
	\]	
\end{lemma}

\begin{proof} For an interval $I\subset\RR$ it will be convenient to denote $D_I$:=the Dirichlet Laplacian in $L^2(I)$,
in particular $D=D_{(0,a)}$. Further define the constants
\[
M:=\tfrac{U^{(m)}(0)}{m!}>0,\quad N:=\big\|\tfrac{U^{(m+1)}}{(m+1)!}\big\|_\infty,
\]
and the function
\[
U_0:\ t\mapsto Mt^m.
\]
Let $\varepsilon\in (0,\frac{M}{2}\big)$, then there exists $\delta>0$ with $\delta N<1$ such that
\begin{equation}
	\label{uu00}
\big|U(t)-U_0(t)\big|\le Nt^{m+1} \text{ for all }t\in(0,\delta).
\end{equation}
In particular, one can choose some $\delta_0\in (0,a)$ such that
$\tfrac{M t^m}{2}\le U(t)\le\tfrac{3M t^m}{2}$ for all $t\in(0,\delta_0)$.
Let $c:=\min\limits_{t\in[\delta_0,a]} U(t)$, then $c>0$ by (i), and
\begin{equation}
	 \label{utl0}
U(t)\ge \min\Big\{\tfrac{M t^m}{2},c\Big\} \text{ for all }t\in(0,a).
\end{equation}

Let $n\in\NN$ be fixed. Let $q>0$ (the precise value will be chosen later). The min-max principle gives the upper bound
\begin{equation}
	\label{upb}
	E_n(D+\alpha U)\le E_n(D_{(0,2\alpha^{-q})}+\alpha U).
\end{equation}
For a lower bound we pick two $C^\infty$-smooth functions $\chi_1,\chi_2:\RR\to \RR$  with
$\chi_1^2+\chi_2^2=1$ such that $\chi_1(t)=1$ for all $t\le 1$ and $\chi_2(t)=1$ for all $t\ge 2$,
and define $\chi_{j,\alpha}:=\chi_j(\alpha^q\,\cdot\,)$. For any $f\in H^1_0(0,a)$ one has
the obvious relation $\|f\|^2_{L^2(0,a)}=\|\chi_{1,\alpha}f\|^2_{L^2(0,2\alpha^{-q})}	+ \|\chi_{2,\alpha}f\|^2_{L^2(\alpha^{-q},a)}$
and the so-called IMS formula
\begin{align*}
\int_0^a \big(|f'|^2+\alpha U|f|^2+W_\alpha|f|^2\big)dt&=\int_0^{2\alpha^{-q}}
\Big(\big|(\chi_{1,\alpha}f)'\big|^2+\alpha U|\chi_{1,\alpha}f|^2\Big)dt\\
&\quad +\int_{\alpha^{-q}}^a
\Big(\big|(\chi_{2,\alpha}f)'\big|^2+\alpha U|\chi_{2,\alpha}f|^2\Big)dt
\end{align*}
with $W_\alpha:=|\chi_{1,\alpha}'|^2+|\chi_{2,\alpha}'|^2$.
The left-hand side is the sesquilinear form for the operator $D+\alpha U+W_\alpha$ computed on $(f,f)$, while
then the right-hand side is the sesquilinear form for the operator
$(D_{(0,2\alpha^{-q})}+\alpha U)\oplus (D_{(\alpha^{-q},a)}+\alpha U)$ computed on $\big( (\chi_{1,\alpha}f,\chi_{2,\alpha}f),(\chi_{1,\alpha}f,\chi_{2,\alpha}f)\big)$. The min-max principle implies
\begin{equation}
	 \label{uta1}
\begin{aligned}
E_n(D+\alpha U+W_\alpha)&\ge E_n\Big((D_{(0,2\alpha^{-q})}+\alpha U)\oplus (D_{(\alpha^{-q},a)}+\alpha U)\Big)\\
&\ge \min\big\{ E_n(D_{(0,2\alpha^{-q})}+\alpha U), E_1(D_{(\alpha^{-q},a)}+\alpha U)\big\}.
\end{aligned}
\end{equation}
By \eqref{utl0} one has $\alpha U(t)\ge \frac{M}{2}\alpha^{1-mq}$ for all $t\in(\alpha^{-q},a)$,
which yields the lower bound
$E_1(D_{(\alpha^{-q},a)}+\alpha U)\ge \frac{M}{2}\alpha^{1-mq}$.

Now let $p>q$, then $\alpha^{-p}\le 2\alpha^{-q}$. Due to \eqref{utl0} for all $t\in(0,\alpha^{-p})$ we have $U(t)\le \frac{3M}{2}\alpha^{1-pm}$, which shows
\begin{align*}
E_n(D_{(0,2\alpha^{-q})}+\alpha U)&\le E_n(D_{(0,\alpha^{-p})}+\alpha U)\le E_n(D_{(0,\alpha^{-p})})+\frac{3M}{2}\alpha^{1-pm}=\dfrac{\pi^2}{n^2}\alpha^{2p}+\frac{3M}{2}\alpha^{1-pm}.
\end{align*}
From now on assume
\begin{equation}
	\label{qqq}
q\in\big(0,\tfrac{1}{m+2}\big)
\end{equation}
and set $p:=\frac{1}{m+2}$, then $E_n(D_{(0,2\alpha^{-q})}+\alpha U)=\cO (\alpha^{\frac{2}{m+2}})< E_1(D_{(\alpha^{-q},a)}+\alpha U)$,
and from \eqref{uta1} we obtain $E_n(D+\alpha U+W_\alpha)\ge E_n(D_{(0,2\alpha^{-q})}+\alpha U)$. Taking into account
$\|W_\alpha\|_\infty=\cO(\alpha^{2q})$ and \eqref{upb} we arrive at
\begin{equation}
	\label{bbb}
E_n(D+\alpha U)= E_n(D_{(0,2\alpha^{-q})}+\alpha U)+\cO(\alpha^{2q}).
\end{equation}
For all $t\in(0,2\alpha^{-q})$ by \eqref{uu00} one has $\big|U(t)-U_0(t)\big|\le 2^{m+1}N\alpha^{-(m+1)q}$, then
\begin{align*}
E_n(D_{(0,2\alpha^{-q})}+\alpha U)&=E_n(D_{(0,2\alpha^{-q})}+\alpha U+\alpha(U-U_0)\big)=E_n(D_{(0,2\alpha^{-q})}+\alpha U_0)+\cO(\alpha^{1-(m+1)q}),
\end{align*}
and the substitution into \eqref{bbb} gives
\begin{equation}
	\label{part1}
E_n(D+\alpha U)= E_n(D_{(0,2\alpha^{-q})}+\alpha U_0)+\cO(\alpha^{2q}+\alpha^{1-(m+1)q}).
\end{equation}

Now we are going to compare the eigenvalues of $D_{(0,2\alpha^{-q})}+\alpha U_0$ with those of $S_{m, M\alpha}$, see \eqref{smb}. The min-max-principle gives the upper bound
\begin{equation}
	\label{part2a}
	E_n(S_{m,M_\alpha})\le E_n(D_{(0,2\alpha^{-q})}+\alpha U_0).
\end{equation}
With the functions $\chi_{j,\alpha}$ and $W_\alpha$ introduced above we have again the identity
\[
\|f\|^2_{L^2(0,+\infty)}=\|\chi_{1,\alpha}f\|^2_{L^2(0,2\alpha^{-q})}	+ \|\chi_{2,\alpha}f\|^2_{L^2(\alpha^{-q},+\infty)}
\]
and
\begin{align*}
	\int_0^{+\infty} \big(|f'|^2+\alpha U_0|f|^2+W_\alpha|f|^2\big)dt&=\int_0^{2\alpha^{-q}}
	\Big(\big|(\chi_{1,\alpha}f)'\big|^2+\alpha U_0|\chi_{1,\alpha}f|^2\Big)dt\\
	&\quad +\int_{\alpha^{-q}}^{+\infty}
	\Big(\big|(\chi_{2,\alpha}f)'\big|^2+\alpha U_0|\chi_{2,\alpha}f|^2\Big)dt
\end{align*}
for all $f\in H^1_0(0,+\infty)$ such that the left-hand side is finite.
The left-hand side is the sesquilinear form for the operator $S_{m,M\alpha}+W_\alpha$ computed on $(f,f)$
and the right-hand side is the sesquilinear form for $(D_{(0,2\alpha^{-q})}+\alpha U_0)\oplus B_\alpha$ computed
on $\big( (\chi_{1,\alpha}f,\chi_{2,\alpha}f),(\chi_{1,\alpha}f,\chi_{2,\alpha}f)\big)$, where $B_\alpha$ is the operator
acting as $f\mapsto-f''+\alpha U_0 f$ in $L^2(\alpha^{-q},+\infty)$ with Dirichlet boundary condition.
Using $\|W_\alpha\|_\infty=\alpha^{2q}$ and $U_0(t)\ge M\alpha^{-mq}$ for all $t\in(\alpha^{-q},+\infty)$ we obtain
$E_1(B_\alpha)\ge M\alpha^{1-mq}$ and then
\begin{equation}
	\label{low44}
\begin{aligned}
	E_n(S_{m,M\alpha})&+\cO(\alpha^{2q})=E_n(S_{m,M\alpha}+W_\alpha)	\ge E_n\big( (D_{(0,2\alpha^{-q})}+\alpha U_0)\oplus B_\alpha\big)\\
	&\ge
	\min\Big\{ E_n(D_{(0,2\alpha^{-q})}+\alpha U_0), E_1(B_\alpha)\Big\}\ge \min\Big\{ E_n(D_{(0,2\alpha^{-q})}+\alpha U_0), M\alpha^{1-mq}\Big\}.
\end{aligned}
\end{equation}
For any $\delta\in(0,2\alpha^{-q})$ one has $U_0(t)\le 2^mM\delta^m$ for all $t\in(0,\delta)$, therefore,
\begin{align*}
E_n(D_{(0,2\alpha^{-q})}+\alpha U_0)&\le E_n(D_{(0,\delta)}+\alpha U_0)\le E_n(D_{(0,\delta)})+2^mM\alpha \delta^{m}=\dfrac{\pi^2n^2}{\delta^2}+2^mM\alpha\delta^m,
\end{align*}
and for $\delta:=\alpha^{-\frac{1}{m+2}}$ this results in $E_n(D_{(0,2\alpha^{-q})}+\alpha U_0)=\cO(\alpha^\frac{2}{m+2})$.
Using \eqref{qqq} and \eqref{low44} yields $E_n(S_{m,M\alpha})\ge E_n(D_{(0,2\alpha^{-q})}+\alpha U_0)+\cO(\alpha^{2q})$, and by combining with \eqref{part2a} we arrive at the asymptotics
$E_n(D_{(0,2\alpha^{-q})}+\alpha U_0)=E_n(S_{m,M\alpha})+\cO(\alpha^{2q})$.
The substitution into \eqref{part1} with $q:=\frac{2}{m+3}$ gives
$E_n(D+\alpha U)= E_n(S_{m,M\alpha})+\cO(\alpha^{\frac{2}{m+3}})$,
and we conclude the proof by using $E_n(S_{m,M\alpha})\stackrel{\eqref{smb}}{=}(\alpha M)^\frac{2}{2+m}E_n(S_{m,1})$.
\end{proof}

\begin{theorem}[Maximum curvature attained at an endpoint of $\Gamma$]\label{kmax}
	Let $m\in\NN$ and assume that:
	\begin{itemize}
		\item the boundary $\partial\Omega$ is $C^{m+3}$-smooth,
		\item the curvature $k$ has its unique global maximum on $[0,\ell]$ at $0$, i.e.
		$k(s)<k_*\equiv k(0)$ for all $s\in(0,\ell]$,
		\item there holds $k^{(m)}(0)<0$ and $k^{(j)}(0)=0$ for all $j\in\{1,\dots,m-1\}$.
	\end{itemize}	
Then for any fixed $n\in\NN$ and  $\alpha\to +\infty$ one has
\begin{equation}
	\label{eqa}
E_n(Q_\alpha)=-\alpha^2-k_*\alpha+ \Big(-\dfrac{k^{(m)}(0)}{m!}\Big)^\frac{2}{m+2} E_n(S_{m,1})\alpha^{\frac{2}{m+2}}+\cO(\alpha^{q}),
\end{equation}
with any $q>\frac{7}{12}$ for $m=1$ and $q=\frac{2}{m+3}$ for if $m\ge 2$,
and the operator $S_{m,1}$ is defined in \eqref{smb}. In particular,
\begin{itemize}
	\item for $m=1$:
\begin{equation}
	 \label{mm1}
E_n(Q_\alpha)=-\alpha^2-k_*\alpha +a_n\big(-k'(0)\big)^\frac{2}{3} \alpha^{\frac{2}{3}}+\cO(\alpha^{q}),
\end{equation}
where $(-a_n)$	is the $n$-th zero of the Airy function $\mathrm{Ai}$ and $q>\frac{7}{12}$ is arbitrary,
\item for $m=2$:
\begin{equation}
	\label{mm2}
E_n(Q_\alpha)=-\alpha^2-k_*\alpha +(4n-1)\sqrt{-\dfrac{k''(0)}{2}} \cdot \sqrt{\alpha}+\cO(\alpha^{\frac{2}{5}}).
\end{equation}
\end{itemize}	
\end{theorem}

\begin{proof}
Let $n\in\NN$ be fixed. We denote again
$U:=k_*-k$,  $V:=\frac{k_*^2-2kk_*-k^2}{4}$,
then $\Lambda'_\alpha=D+\alpha U+V$, where $D$	is the Dirichlet Laplacian in $L^2(0,\ell)$,  and note that $V$ is bounded and independent of $\alpha$. The function $U$ satisfies the assumptions of Lemma \ref{semikl} on $(0,\ell)$, with $U^{(m)}(0)=-k^{(m)}(0)$, so we obtain
\begin{align*}
E_n(\Lambda'_\alpha)&=E_n(D+\alpha U)+\cO(1)
=\Big(-\dfrac{k^{(m)}(0)}{m!}\Big)^{\frac{2}{m+2}} E_n(S_{m,1})\alpha^{\frac{2}{m+2}}+\cO(\alpha^{\frac{2}{m+3}}).
\end{align*}
The substitution into Theorem \ref{thmeff}(ii) gives, for any $\vartheta\in (\frac{1}{6},\frac{1}{4})$,
\begin{align*}
	E_n(Q_\alpha)&=-\alpha^2-k_*\alpha+\Big(-\dfrac{k^{(m)}(0)}{m!}\Big)^{\frac{2}{m+2}} E_n(S_{m,1})\alpha^{\frac{2}{m+2}}+R_\alpha\\
\text{with }R_\alpha&=\cO\Big(\alpha^{\frac{2}{m+3}}+\alpha^{-\vartheta}\big(
\alpha^{-\frac{1}{2}}\alpha^{\frac{2}{m+3}}+1\big)\alpha^{\frac{2}{2+m}}\Big).
\end{align*}
For $m=1$ one has $R_\alpha= \cO(\alpha^{\frac{1}{2}}+\alpha^{\frac{1}{6}-\vartheta}\alpha^{\frac{2}{3}})$, and for $\vartheta$ close to $\frac{1}{4}$ one obtains $R_\alpha=\cO(\alpha^q)$ for any $q>\frac{7}{12}$.
If $m\ge 2$, then $\alpha^{-\frac{1}{2}}\alpha^{\frac{2}{2+m}}+1=\cO(1)$, and then $R_\alpha=\cO(\alpha^{\frac{2}{m+3}}+\alpha^{\frac{2}{2+m}-\vartheta})=\cO(\alpha^\frac{2}{m+3})$ as $\vartheta\ge \frac{2}{(m+2)(m+3)}\in(0,\frac{1}{10}]$.	
\end{proof}

\end{document}